\documentclass[11pt,french,english,a4paper]{article}
\usepackage[T1]{fontenc}
\usepackage{amsmath}
\usepackage[dvips]{graphicx}
\usepackage[symbol]{footmisc}

\usepackage{amsmath,amsfonts}

\usepackage{shadow}

\usepackage{latexsym}

\usepackage{latexsym}

\usepackage{textcomp}

\usepackage{indentfirst}

\usepackage{mathrsfs}

\usepackage{geometry}

\allowdisplaybreaks

\let\tilde\widetilde
\def\limnto{\mathrel{\mathop{\longrightarrow\kern 0pt}\limits_{n\to\infty}}}

\def\1{\mbox{1\hspace{-.25em}I}}

\let\tilde\widetilde

\newcommand{\theoremname}{Theorem}
\newcommand{\definitionname}{Definition}
\newcommand{\propositionname}{Proposition}
\newcommand{\lemmaname}{Lemma}
\newcommand{\corollaryname}{Corollary}
\newcommand{\propertyname}{Property}
\newcommand{\exercisename}{Exercise}
\newcommand{\remarkname}{Remark}
\newcommand{\recallname}{Recall}
\newcommand{\notationname}{Notation}

\newtheorem{theorem}{\theoremname}[section]
\newtheorem{lemma}[theorem]{\lemmaname}
\newtheorem{corollary}[theorem]{\corollaryname}

\newtheorem{definition}[theorem]{\definitionname}

\newtheorem{remark}{\remarkname}

\newtheorem{proposition}[theorem]{\propositionname}

\usepackage{babel}
\addto\extrasfrench{
\providecommand{\fg}{\ifdim\lastskip>\z@\unskip\fi~\frqq}}

\begin{document}

\title{\large{REFLECTED BACKWARD STOCHASTIC DIFFERENTIAL EQUATION WITH JUMPS AND VISCOSITY SOLUTION OF SECOND ORDER INTEGRO-DIFFERENTIAL EQUATION WITHOUT MONOTONICITY CONDITION}: \large{CASE WITH THE MEASURE OF L\'EVY INFINITE}}
\date{December, 21, 2017}
\selectlanguage{english}

\author{\textsc{L.\ SYLLA\footnotemark[2]}}

\footnotetext[2]{Universit\'e Gaston Berger, LERSTAD, CEAMITIC, e-mail: sylla.lamine@ugb.edu.sn}
\maketitle

\begin{abstract}
We consider the problem of viscosity solution of integro-partial differential equation(IPDE in short) with one obstacle via the solution of reflected backward stochastic differential  equations(RBSDE in short) with jumps. We show  existence and uniqueness of a continuous viscosity solution of equation with non local terms, in case the generator is not monotonous and Levy's measure is infinite.\\ 
\end{abstract}
\bigskip{}


\bigskip{}

\newpage
\section{Introduction}
We consider the following system of integro-partial differential equation with one-obstacle $\ell$, which is a function of $(t,x)$: $\forall i\in\{1,\ldots,m\}$,
\begin{equation}\label{eq1}
\left
\{\begin{array}{ll}
\min\Big\{ u^{i}(t,x)-\ell(t,x);-\partial_{t}u^{i}(t,x)-b(t,x)^{\top}\mathrm{D}_{x}u^{i}(t,x)-\frac{1}{2}\mathrm{Tr}(\sigma\sigma^{\top}(t,x)\mathrm{D}^{2}_{xx}u^{i}(t,x))\\
\quad\quad-\mathrm{K}_{i}u^{i}(t,x)-\mathit{h}^{(i)}(t,x,u^{i}(t,x),(\sigma^{\top}\mathrm{D}_{x}u^{i})(t,x),\mathrm{B}_{i}u^{i}(t,x))\Big\}=0,\quad (t,x)\in\left[ 0,T\right] \times\mathbb{R}^{k};\\
u^{i}(T,x)=g^{i}(x);
\end{array}
\right.
\end{equation}
where the operators $\mathrm{B}_{i}$ and $\mathrm{K}_{i}$ are defined as follows:
\begin{eqnarray}
\mathrm{B}_{i}u^{i}(t,x) & = & \displaystyle\int_{\mathrm{E}}\gamma^{i}(t,x,e)(u^{i}(t,x+\beta(t,x,e))-u^{i}(t,x))\lambda(\mathrm{d} e);\label{2.2}\\
\mathrm{K}_{i}u^{i}(t,x) & = & \displaystyle\int_{\mathrm{E}}(u^{i}(t,x+\beta(t,x,e))-u^{i}(t,x)-\beta(t,x,e)^{\top}\mathrm{D}_{x}u^{i}(t,x))\lambda(de).\nonumber
\end{eqnarray}
The resolution of (\ref{eq1}) is in connection with the following system of backward stochastic differential equations with jumps and one-obstacle $\ell$:

\begin{equation}\label{eq2}
\left
\{\begin{array}{ll}
(i)~dY^{i;t,x}_{s}=-f^{(i)}(s,X^{t,x}_{s},(Y^{i;t,x}_{s})_{i=1,m},Z^{i;t,x}_{s},U^{i;t,x}_{s})ds-
\mathrm{d}\mathrm{K}^{i;t,x}_{s}\\
\quad\quad\quad\quad\quad\quad
\quad\quad+Z^{i;t,x}_{s}\mathrm{d}
\mathrm{B}_{s}+\displaystyle\int_{\mathrm{E}}\mathrm{U}^{i;t,x}
_{s}(e)\tilde{\mu}(\mathrm{d}s,\mathrm{d}e),\quad s\leq T;\\
(ii)~Y^{i;t,x}_{s}\geq \ell(s,X^{t,x}_{s})~\textrm{and}~ \displaystyle\int^{T}_{0}(Y^{i;t,x}_{s}- \ell(s,X^{t,x}_{s}))\mathrm{d}\mathrm{K}^{i;t,x}_{s}=0; 
\end{array}
\right.
\end{equation}
and\\
the following standard stochastic differential equation of diffusion-jump type:
\begin{equation}\label{2.4}
X^{t,x}_{s}=x+\displaystyle\int^{s}_{t}b(r,X^{t,x}_{r})\, \mathrm{d}r+\displaystyle\int^{s}_{t}\sigma(r,X^{t,x}_{r})\, \mathrm{d}B_{r}+\displaystyle\int^{s}_{t}
\displaystyle\int_{E}\beta(r,X^{t,x}_{r-},e)\tilde{\mu}(dr,de),
\end{equation}
for $s\in[t,T]$ and $X^{t,x}_{s}=x$ if $s\leq t$.\\

It is recalled that pioneering work was done for the resolution of (\ref{eq1}), among these works we can mention those of Barles and al. \cite{bar} in case without obstacle, Harraj and al. \cite{har} in the case with two obstacles; with as common point the hypothesis of monotony on the generator and $\gamma \geq 0$.
But recently Hamad\`ene and Morlais relaxed these conditions with $\lambda(.)$ finite \cite{hamaMor}.\\
In this work we propose to solve (\ref{eq1}) by relaxing the monotonicity of the generator and the positivity of $\gamma$ and assuming that $\lambda=\infty$.\\
Our paper is organized as follows: in the next section we give the notations and the assumptions of our objects; in section $3$ we recall a number of existing results; in section $4$ we build estimates and properties for a good resolution of our problem;  section $5$ is reserved to give our main result and the section $6$ for doing an extension of our result.\\
And in the end, classical definition of the concept of viscosity solution is put in appendix.

\section{Notations and assumptions}
Let $\left(\Omega,\mathcal{F},(\mathcal{F}_{t})_{t\leq T},\mathbb{P}\right)$ be a stochastic basis such that $\mathcal{F}_{0}$ contains all $\mathbb{P}-$null sets of $\mathcal{F}$, and $\mathcal{F}_{t}=\mathcal{F}_{t+}:=\bigcap_{\epsilon>0}
\mathcal{F}_{t+\epsilon},~t\geq 0$, and we suppose that the filtration is generated by the two mutually independents processes:\\
(i) $B:=(B_{t})_{t\geq 0}$ a $d$-dimensional Brownian motion and,\\
(ii) a Poisson random measure $\mu$ on $\mathbb{R}^{+}\times\mathrm{E}$ where $\mathrm{E}:=\mathbb{R}^{\ell}-\{0\}$ is equipped with its Borel field $\mathcal{E}$ $(\ell\geq 1)$. The compensator $\nu(\mathrm{d}t,\mathrm{d}e)=\mathrm{d}t\lambda(\mathrm{d}e)$ is such that $\{\tilde{\mu}(\left[0,t\right]\times A)=(\mu-\lambda)(\left[ 0,t\right]\times A)\}_{t\geq 0}$ is a martingale for all $A\in\mathcal{E}$ satisfying $\lambda(A)<\infty$. We also assume that $\lambda$ is a $\sigma$-finite measure on $(E,\mathcal{E})$, integrates the function $(1\wedge\mid e\mid ^{2})$ and $\lambda(E)=\infty$.\\
Let's now introduce the following spaces:\\ 
(iii) $\mathcal{P}~(resp.~\mathbf{P})$ the field on $\left[0,T\right]\times \Omega$ of $\mathcal{F}_{t\leq T}$-progressively measurable (resp. predictable) sets.\\
(iv) For $\kappa\geq 1$, $\mathbb{L}^{2}_{\kappa}(\lambda)$ the space of Borel measurable functions $\varphi:=(\varphi(e))_{e\in E}$ from $E$ into $\mathbb{R}^{\kappa}$ such that 
$\|\varphi\|^{2}_{\mathbb{L}^{2}_{\kappa}(\lambda)}=\displaystyle\int_{E}\left|\varphi(e)\right|^{2}_{\kappa}\lambda(\mathrm{d}e)<\infty$; $\mathbb{L}^{2}_{1}(\lambda)$ will be simply denoted by $\mathbb{L}^{2}(\lambda)$;\\
(v) $\mathcal{S}^{2}(\mathbb{R}^{\kappa})$ the space of RCLL (for right continuous with left limits) $\mathcal{P}$-measurable and $\mathbb{R}^{\kappa}$-valued processes such that $\mathbb{E}[\sup_{s\leq T} \left|Y_{s}\right|^{2}]<\infty$; $\mathcal{A}^{2}_{c}$ is its subspace of continuous non-decreasing processes $(\mathrm{K}_{t})_{t\leq T}$ such that 
$\mathrm{K}_{0}=0$ and $\mathbb{E}\left[(\mathrm{K}_{T})^{2} \right]<\infty$;\\
(vi) $\mathbb{H}^{2}(\mathbb{R}^{\kappa\times d})$ the space of processes $Z:=(Z_{s})_{s\leq T}$ which are $\mathcal{P}$-measurable, $\mathbb{R}^{\kappa\times d}$-valued and satisfying $\mathbb{E}\left[\displaystyle\int^{T}_{0}\left|Z_{s}\right|^{2}\, \mathrm{d} s\right]<\infty$;\\
(vii) $\mathbb{H}^{2}(\mathbb{L}^{2}_{\kappa}(\lambda))$ the space of processes $U:=(U_{s})_{s\leq T}$ which are $\mathbf{P}$-measurable, $\mathbb{L}^{2}_{\kappa}(\lambda)$-valued and satisfying $\mathbb{E}\left[\displaystyle\int^{T}_{0}\|U_{s}(\omega)\|^{2}_{\mathbb{L}^{2}_{\kappa}(\lambda)}\, \mathrm{d} s\right]<\infty$;\\
(viii) $\Pi_{g}$ the set of deterministics functions\\ $\varpi:~(t,x)\in [0,T]\times \mathbb{R}^{\kappa}\mapsto\varpi(t,x)\in\mathbb{R}$ of polynomial growth, i.e., for which there exists two non-negative constants $C$ and $p$ such that for any $(t,x)\in [0,T]\times \mathbb{R}^{\kappa}$,
$$\left|\varpi(t,x)\right|\leq C(1+\left|x\right|^{p}).$$
The subspace of $\Pi_{g}$ of continuous functions will be denoted by $\Pi^{c}_{g}$;\\
(ix) $\mathcal{U}$ the subclass of $\Pi^{c}_{g}$ which consists of functions \\$\Phi:~(t,x)\in [0,T]\times \mathbb{R}^{\kappa}\mapsto\mathbb{R}$ such that for some non-negative constants $C$ and $p$ we have
$$\left|\Phi(t,x)-\Phi(t,x')\right|\leq C(1+\left|x\right|^{p}+\left|x'\right|^{p})\left|x-x'\right|,~\textrm{for any}~t,~x,~x'.$$
(x) For any process $\theta:=(\theta_{s})_{s\leq T}$ and $t\in(0,T],~\theta_{t-}=\lim_{s\nearrow t}\theta_{s}$ and\\
$$\Delta_{t}\theta=\theta_{t}-\theta_{t-}.$$
Now let $b$ and $\sigma$  be the following functions:
$$b:(t,x)\in [0,T]\times \mathbb{R}^{k}\mapsto b(t,x)\in\mathbb{R}^{k};$$
$$\sigma:(t,x)\in [0,T]\times \mathbb{R}^{k}\mapsto\sigma(t,x)\in\mathbb{R}^{k\times d}.$$
We assume that they are jointly continuous in $(t,x)$ and Lipschitz continuous w.r.t. $x$ uniformly in $t$, i.e., there exists a constant $C$ such that,
\begin{equation}\label{2.5}
\forall (t,x,x')\in[0,T]\times \mathbb{R}^{k+k},~\left|b(t,x)-b(t,x')\right|+\left|\sigma(t,x)-\sigma(t,x')\right|\leq C\left|x-x'\right|.
\end{equation}
Let us notice that by (\ref{2.5}) and continuity, the functions $b$ and $\sigma$ are of linear growth, i.e., there exists a constant $C$ such that
\begin{equation}\label{2.6}
\forall (t,x,x')\in[0,T]\times \mathbb{R}^{k+k},~\left|b(t,x)\right|+\left|\sigma(t,x)\right|\leq C\left|1+x\right|.
\end{equation}
Let $\beta:(t,x,e)\in [0,T]\times \mathbb{R}^{k}\times E\mapsto \beta(t,x,e)\in\mathbb{R}^{k}$ be a measurable function such that for some real constant $C$, and for all $e\in E$,
\begin{eqnarray}\label{2.7}
& (i) & \left|\beta(t,x,e)\right|  \leq  C(1\wedge\left|e\right|); \\ \nonumber  
& (ii) & \left|\beta(t,x,e)-\beta(t,x',e)\right|\leq C\left|x-x'\right|(1\wedge\left|e\right|);\\
& (iii) & \textrm{the mapping}~(t,x)\in[0,T]\times \mathbb{R}^{k}\mapsto \beta(t,x,e)\in\mathbb{R}^{k}~\textrm{is continuous for any}~\mathrm{e}\in\mathrm{E}\nonumber.
\end{eqnarray}
We are now going to introduce the objects which are specifically connected to the RBSDE with jumps we will deal with. Let $\ell$  the barrier of (\ref{eq2}); $(g^{i})_{i=1,m}$ and $(h^{(i)})_{i=1,m}$ be two functions defined as follows: for $i=1,\ldots,m$,
\begin{eqnarray}
g^{i}:\mathbb{R}^{k} & \longrightarrow & \mathbb{R}^{m}\nonumber \\
{}{}x & \longmapsto & g^{i}(x)\nonumber
\end{eqnarray}
and
\begin{eqnarray}
h^{(i)}:[0,T]\times\mathbb{R}^{k+m+d+1} & \longrightarrow & \mathbb{R}\nonumber \\
{}{}(t,x,y,z,q) & \longmapsto & h^{(i)}(t,x,y,z,q).\nonumber
\end{eqnarray}
Moreover we assume they satisfy:\\
(\textbf{H1})\label{2.10}: The reflecting barrier $\ell$ is real valued and $\mathcal{P}$-measurable process satisfying, $\ell\in\mathcal{U}$ i.e., it is continuous and there exists constants $C$ and $p$ such that,\\
$\left|\ell(t,x)-\ell(t,x')\right|\leq C(1+\left|x\right|^{p}+\left|x'\right|^{p})\left|x-x'\right|$, for any $t\geq 0$, $x$, $x'$.\\
\\
(\textbf{H2}): For any $i\in\left\lbrace 1,\ldots,m\right\rbrace$, the function $g^{i}$ belongs to $\mathcal{U}$.\\
\\
(\textbf{H3}): For any $i\in\left\lbrace 1,\ldots,m\right\rbrace$,
\begin{eqnarray} 
& (i) &~\textrm{the function}~h^{(i)}~\textrm{is Lipschitz in}~ (y,z,q)~\textrm{uniformly in}~(t,x),~\textrm{i.e., there exists a real constant}\nonumber\\
& {}{} & \textrm{C such that for any}~ (t,x)\in[0,T]\times\mathbb{R}^{k}, (y,z,q)~\textrm{and}~(y',z',q')~\textrm{elements of}~\mathbb{R}^{m+d+1},\nonumber\\
& {}{} &\left|h^{(i)}(t,x,y,z,q)-h^{(i)}(t,x,y',z',q')\right|\leq C(\left|y-y'\right|+\left|z-z'\right|+\left|q-q'\right|);\label{2.11}\\
& (ii) & ~\textrm{the}~(t,x)\mapsto h^{(i)}(t,x,y,z,q), ~\textrm{for fixed}~ (y,z,q)\in\mathbb{R}^{m+d+1},~\textrm{belongs uniformly to}~\mathcal{U},~\textrm{i.e., it}\nonumber\\
& {}{} &\textrm{is continuous and there exists constants C and p (which do not depend on}~(y,z,q))~\textrm{such that},\nonumber\\
& {} & \left|h^{(i)}(t,x,y,z,q)-h^{(i)}(t,x',y,z,q)\right|\leq C(1+\left|x\right|^{p}+\left|x'\right|^{p})\left|x-x'\right|,~\textrm{for any}~t\geq 0,~x,~x'.
\end{eqnarray}
Next let $\gamma^{i},~i=1,\ldots,m$ be Borel measurable functions defined from $[0,T]\times\mathbb{R}^{k}\times E$ into $\mathbb{R}$ and satisfying:
\begin{eqnarray}\label{2.12} 
& (i) &\left|\gamma^{i}(t,x,e)\right|\leq C(1\wedge\left|e\right|);\nonumber\\
& (ii) & \left|\gamma^{i}(t,x,e)-\gamma^{i}(t,x',e)\right|\leq C(1\wedge\left|e\right|)\left|x-x'\right|(1+\left|x\right|^{p}+\left|x'\right|^{p});\\
& (iii) & \textrm{the mapping}~t\in[0,T]\mapsto \gamma^{i}(t,x,e)\in\mathbb{R}~ \textrm{is continuous for any}~(x,e)\in\mathbb{R}^{k}\times E.\nonumber
\end{eqnarray}
Finally we introduce the following functions $(f^{(i)})_{i=1,m}$ defined by:
\begin{equation}\label{2.13}
\forall (t,x,y,z,\zeta)\in[0,T]\times\mathbb{R}^{k+m+d}\times \mathbb{L}^{2}(\lambda),~f^{(i)}(t,x,y,z,\zeta):=h^{(i)}\left(t,x,y,z,\displaystyle\int_{E}\gamma^{i}(t,x,e)\zeta(e)\lambda(de)\right).
\end{equation}
The functions $(f^{(i)})_{i=1,m}$, enjoy the two following properties:

\begin{eqnarray}\label{2.15} 
& (a) &~\textrm{The function}~f^{(i)}~\textrm{is Lipschitz in}~ (y,z,\zeta)~\textrm{uniformly in}~(t,x),~\textrm{i.e., there exists a real constant}\nonumber\\
& {}{} & \textrm{C such that}\nonumber\\
& {}{} &\left|f^{(i)}(t,x,y,z,\zeta)-f^{(i)}(t,x,y',z',\zeta')\right|\leq C(\left|y-y'\right|+\left|z-z'\right|+\|\zeta-\zeta'\|_{\mathbb{L}^{2}(\lambda)});\\
\nonumber
& {}{} & \textrm{since}~h^{(i)}~\textrm{is uniformly Lipschitz in}~(y,z,q)~\textrm{and}~ \gamma^{i}~\textrm{verifies (\ref{2.11})-(i)};\nonumber\\
& (b) &~\textrm{The function}~(t,x)  \in[0,T]\times\mathbb{R}^{k}\mapsto f^{(i)}(t,x,0,0,0)~\textrm{belongs}~to~\Pi^{c}_{g};\nonumber\\
\nonumber
& {}{} & \textrm{and then}~ \mathbb{E}\left[\displaystyle\int^{T}_{0}\left|f^{(i)}(r,X^{t,x}_{r},0,0,0)\right|^{2}\,dr \right]<\infty.
\end{eqnarray}
Having defined our data and put our assumptions, we can look at the state of the art.
 
\section{Preliminaires}
\subsection{A class of diffusion processes with jumps}
Let $(t,x)\in[0,T]\times\mathbb{R}^{d}$ and $(X^{t,x}_{s})_{s\leq T}$ be the stochastic process solution of (\ref{2.4}).
Under assumptions (\ref{2.5})-(\ref{2.7}) the solution of Equation (\ref{2.4}) exists and is unique (see \cite{fuj} for more details).
We state some properties of the process $\{(X^{t,x}_{s}),~s\in[0,T]\}$ which can found in \cite{fuj}. 
\begin{proposition}
For each $t\geq 0$, there exists a version of $\{(X^{t,x}_{s}),~s\in[t,T]\}$ such that $s\rightarrow X^{t}_{s}$ is a $C^{2}(\mathbb{R}^{d})$-valued rcll process. Moreover it satisfies the following estimates: $\forall p\geq 2,~x,x'\in\mathbb{R}^{d}$ and $s\geq t$,
\begin{eqnarray}
\mathbb{E}[\sup_{t\leq r\leq s} \left|X^{t,x}_{r}-x\right|^{p}] & \leq &  M_{p}(s-t)(1+\left|x\right|^{p});\nonumber\\
\mathbb{E}[\sup_{t\leq r\leq s} \left|X^{t,x}_{r}-X^{t,x'}_{r}-(x-x')^{p}\right|^{p}] & \leq &  M_{p}(s-t)(\left|x-x'\right|^{p});\label{3.16}
\end{eqnarray}
for some constant $M_{p}$.
\end{proposition}
\subsection{Existence and uniqueness for a RBSDE with jumps}
Let $(t,x)\in[0,T]\times\mathbb{R}^{d}$  and we consider the following m-dimensional RBSDE with jumps:
\begin{equation}\label{3.17}
\left
\{\begin{array}{ll}
(i)~\vec{Y}^{t,x}:=(Y^{i,t,x})_{i=1,m}\in\mathcal{S}^{2}(\mathbb{R}^{m}),~Z^{t,x}:=(Z^{i,t,x})_{i=1,m}\in\mathbb{H}^{2}(\mathbb{R}^{m\times d}),\\
 K^{t,x}:=(K^{i,t,x})_{i=1,m}\in\mathcal{A}^{2}_{c},~ U^{t,x}:=(U^{i,t,x})_{i=1,m}\in\mathbb{H}^{2}(\mathbb{L}^{2}_{m}(\lambda));\\
(ii)~dY^{i;t,x}_{s}=-f^{(i)}(s,X^{t,x}_{s},(Y^{i;t,x}_{s})_{i=1,m},Z^{i;t,x}_{s},U^{i;t,x}_{s})ds-
\mathrm{d}\mathrm{K}^{i;t,x}_{s}\\
\quad\quad\quad\quad\quad\quad
\quad\quad+Z^{i;t,x}_{s}\mathrm{d}
\mathrm{B}_{s}+\displaystyle\int_{\mathrm{E}}\mathrm{U}^{i;t,x}
_{s}(e)\tilde{\mu}(\mathrm{d}s,\mathrm{d}e),\quad s\leq T;\\
(iii)~Y^{i;t,x}_{s}\geq \ell(s,X^{t,x}_{s})~\textrm{and}~ \displaystyle\int^{T}_{0}(Y^{i;t,x}_{s}- \ell(s,X^{t,x}_{s}))\mathrm{d}\mathrm{K}^{i;t,x}_{s}=0; 
\end{array}
\right.
\end{equation}
where for any $i\in\{1,\ldots,m\}$, $Z^{i;t,x}_{s}$ is the ith row of $Z^{t,x}_{s}$, $K^{i;t,x}_{s}$ is the ith component of $K^{t,x}_{s}$ and $U^{i;t,x}_{s}$ is the ith component of $U^{t,x}_{s}$.\\
The following result is related to existence and uniqueness of a solution for the RBSDE with jumps (\ref{3.17}).\\
Its proof is given in \cite{hamaOuk} by using the penalization method (see p .5-12) and the Snell envelope method (see p. 14-16).
\begin{proposition}
Assume that assumptions $(\mathbf{H1}),~(\mathbf{H2})~\text{and }(\mathbf{H3})$ hold. Then for any $(t,x)\in[0,T]\times\mathbb{R}^{d}$, the RBSDE (\ref{3.17}) has an unique solution $(\vec{Y}^{t,x},Z^{t,x},U^{t,x},K^{t,x})$.
\end{proposition}
\begin{remark}
The solution of this RBSDE with jumps exist and is unique since:
\begin{eqnarray*} 
& (i) &\mathbb{E}\left[\left|g(X^{t,x}_{T})\right|^{2}\right]<\infty,~\textrm{due to polynomial growth of}~g~\textrm{and estimate}~(\ref{3.16})~on~X^{t,x};\nonumber\\
& (ii) & \textrm{for any}~i=1,\ldots,m,~f^{(i)}~\textrm{verifies the properties (a)-(b) related to uniform}\\ 
& {}{} & \textrm{Lipschitz w.r.t.}~(y,z,\zeta)~\textrm{and}~ds\otimes d\mathbb{P}-\textrm{square integrability of the process }\\
& {}{} & (f^{(i)}(s,X^{t,x}_{s},0,0,0))_{s\leq T}.
\end{eqnarray*}
\end{remark}
\subsection{Viscosity solutions of integro-differential partial equation with one obstacle}
The following result on one obstacle is proved with two obstacles in Harraj and al. (see. \cite{har}, Theorem 4.6. p. 47 by using (4.1) p. 44), establishes the relationship between the solution of (\ref{3.17}) and the one of system (\ref{eq1}).
\begin{proposition}
Assume that (\textbf{H1}), (\textbf{H2}), (\textbf{H3}) are fulfilled. Then there exists deterministic continuous functions $(u^{i}(t,x))_{i=1,m}$ which belong to $\Pi_{g}$ such that for any $(t,x)\in[0,T]\times\mathbb{R}^{k}$, the solution of the RBSDE (\ref{3.17}) verifies:
\begin{equation}\label{3.18}
\forall i\in\{1,\ldots,m\},~\forall s\in[t,T],~Y^{i;t,x}_{s}=u^{i}(s,X^{t,x}_{s}).
\end{equation} 
Moreover if for any $i\in\{1,\ldots,m\}$,
\begin{eqnarray*} 
& (i) & \gamma^{i}\geq 0;\\
& (ii) & \textrm{for any fixed}~(t,x,\vec{y},z)\in[0,T]\times \mathbb{R}^{k+m+d},~\textrm{the mapping}\\ 
& {}{} & (q\in\mathbb{R})\longmapsto h^{(i)}(t,x,\vec{y},z,q)\in\mathbb{R}~\textrm{is non-decreasing}. 
\end{eqnarray*}
\end{proposition}
The function $(u^{i})_{i=1,m}$ is a continuous viscosity solution (in Barles and al. 's sense, see Definition $6.1$ in the Appendix) of (\ref{eq1}).\\
For the proof see \cite{har} for the same way.\\
Finally, the solution $(u^{i})_{i=1,m}$ of (\ref{eq1}) is unique in the class $\Pi^{c}_{g}$.
\begin{remark}(see \cite{har})
Under the assumptions (\textbf{H1}), (\textbf{H2}), (\textbf{H3}), there exists a unique viscosity solution of (\ref{eq1}) in the class of functions satisfying
\begin{equation}
\lim\limits_{\left|x\right| \to +\infty}\left|u(t,x)\right| e^{-\tilde{A}[\log(\left|x\right|)]^{2}}=0
\end{equation} 
uniformly for $t\in[0,T]$, for some $\tilde{A}>0$.
\end{remark}
\section{Estimates and properties}
In this section we provide estimates for the functions $(u^{i})_{i=1,m}$ defined in (\ref{3.18}). Recall that,\\ $(\vec{Y}^{t,x},Z^{t,x},U^{t,x},K^{t,x}):=((Y^{i;t,x})_{i=1,m},
(Z^{i;t,x})_{i=1,m},(U^{i;t,x})_{i=1,m},(K^{i;t,x})_{i=1,m})$ is the unique solution of the RBSDE with jumps (\ref{3.17}).
\begin{lemma}
Under assumption (\textbf{H1}), (\textbf{H2}), (\textbf{H3}), for any $p\geq 2$ there exists two non-negative constants $C$ and $\rho$ such that,
\begin{equation}\label{4.22}
\mathbb{E}\left[\left\lbrace \displaystyle\int^{T}_{0} ds \left(\displaystyle\int_{E}\left|U^{t,x}_{s}(e)\right|^{2}\lambda(de)\right)\right\rbrace^{\frac{p}{2}}\right]=\mathbb{E}\left[\left\lbrace \displaystyle\int^{T}_{0} ds\|U^{t,x}_{s}\|^{2}_{\mathbb{L}^{2}_{m}(\lambda)}\right\rbrace^{\frac{p}{2}}\right]\leq C\left(1+\left|x\right|^{\rho}\right). 
\end{equation} 
\end{lemma}
\begin{proof} 
First let us point out that since $X^{t,x}_{s}=x$ for $s\in[0,t]$ then, the uniqueness of the solution of RBSDE of (\ref{3.17}) implies that,
\begin{equation}\label{4.23}
Z^{t,x}_{s}=U^{t,x}_{s}=K^{t,x}_{s}=0,~ds\otimes d\mathbf{P}-\textrm{a.e. on}~[0,t]\times\Omega. 
\end{equation}
Next let $p\geq 2$ be fixed. Using the representation (\ref{3.18}), for any $i\in\{1,\ldots,m\}$ and $s\in[t,T]$ we have,
\begin{eqnarray}\label{4.24}
Y^{i;t,x}_{s} & = & g^{i}(X^{t,x}_{T})+\displaystyle\int^{T}_{s}f^{(i)}(r,X^{t,x}_{r},(u^{j}(X^{t,x}_{r}))_{j=1,m},Z^{i;t,x}_{r},U^{i;t,x}_{s})dr+
\mathrm{K}^{i;t,x}_{T}-\mathrm{K}^{i;t,x}_{s}\nonumber\\
&{}{}&-\displaystyle\int^{T}_{s}Z^{i;t,x}_{r}\mathrm{d}
\mathrm{B}_{r}-\displaystyle\int^{T}_{s}\displaystyle\int_{\mathrm{E}}\mathrm{U}^{i;t,x}
_{r}(e)\tilde{\mu}(\mathrm{d}r,\mathrm{d}e).
\end{eqnarray}
This implies that the system of RBSDEs with jumps (\ref{3.17}) turns into a decoupled one since the equations in (\ref{4.24}) are not related each other.\\
Next for any $i=1,\ldots,m$, the functions $u^{i}$, $g^{i}$ and $(t,x)\mapsto f^{(i)}(t,x,0,0,0)$ are of polynomial growth and finally $y\mapsto f^{(i)}(t,x,y,0,0)$ is Lipschitz uniformly w.r.t. $(t,x)$. Then for some $C$ and $\rho\geq 0$,
\begin{equation}\label{4.25}
\mathbb{E}\left[\left|g^{i}(X^{t,x}_{T})\right|^{p}+\left( \displaystyle\int^{T}_{s}\left|f^{(i)}(r,X^{t,x}_{r},(u^{j}(X^{t,x}_{r}))_{j=1,m},0,0)\right|^{2}dr\right)^{\frac{p}{2}}\right]\leq C\left(1+\left|x\right|^{\rho}\right). 
\end{equation} 
Let us now fix $i_{0}\in\{1,\ldots,m\}$,
$\forall s\in[t,T]$,
\begin{equation}\label{4.26}
\left
\{\begin{array}{ll}
(i)~Y^{i_{0},t,x}_{s}\in\mathcal{S}^{2}(\mathbb{R}),~Z^{i_{0},t,x}_{s}\in\mathbb{H}^{2}(\mathbb{R}^{d}),~
 K^{i_{0},t,x}_{s}\in\mathcal{A}^{2}_{c},~U^{i_{0},t,x}_{s}\in\mathbb{H}^{2}(\mathbb{L}^{2}(\lambda));\\
(ii)~Y^{i_{0},t,x}_{s}= g^{i_{0}}(X^{t,x}_{T})+\displaystyle\int^{T}_{s}f^{(i_{0})}(r,X^{t,x}_{r},(u^{j}(X^{t,x}_{r}))_{j=1,m},Z^{i_{0},t,x}_{r},U^{i_{0},t,x}_{r})dr+
K^{i_{0},t,x}_{T}-K^{i_{0},t,x}_{s}\nonumber\\
\qquad\qquad-\displaystyle\int^{T}_{s}Z^{i_{0},t,x}_{r}\mathrm{d}
\mathrm{B}_{r}-\displaystyle\int^{T}_{s}
\displaystyle\int_{\mathrm{E}}U^{i_{0},t,x}_{r}(e)\tilde{\mu}(\mathrm{d}r,\mathrm{d}e).\\
(iii)~Y^{i_{0},t,x}_{s}\geq \ell(s,X^{t,x}_{s})~\textrm{and}~ \displaystyle\int^{T}_{0}(Y^{i_{0},t,x}_{s}- \ell(s,X^{t,x}_{s}))\mathrm{d}K^{i_{0},t,x}_{s}=0. 
\end{array}
\right.
\end{equation}
Applying It\^o formula to $\left|Y^{i_{0},t,x}_{s}\right|^{2}$ between $s$ and $T$, we have
\begin{eqnarray}\label{4.24bis}
& {}{} &\left|Y^{i_{0},t,x}_{s}\right|^{2}+\displaystyle\int^{T}_{s}\left|Z^{i_{0},t,x}_{r}\right|^{2}\,dr+\sum_{s\leq r\leq T}(\Delta Y^{i_{0},t,x}_{r})^{2}\\
& {}{} &=\left|g^{i_{0}}(X^{t,x}_{T})\right|^{2}+2\displaystyle\int^{T}_{s}
Y^{i_{0},t,x}_{r}f^{(i_{0})}(r,X^{t,x}_{r},(u^{j}(X^{t,x}_{r}))_{j=1,m},Z^{i_{0},t,x}_{r},U^{i_{0},t,x}_{r})\,dr+2\displaystyle\int^{T}_{s}Y^{i_{0},t,x}_{r}\,dK^{i_{0},t,x}_{r}\nonumber \\
& {}{} &-2\displaystyle\int^{T}_{s}
\displaystyle\int_{\mathrm{E}}Y^{i_{0},t,x}_{r}U^{i_{0},t,x}_{r}\tilde{\mu}(\mathrm{d}r,\mathrm{d}e)-2\displaystyle\int^{T}_{s}Y^{i_{0},t,x}_{r}Z^{i_{0},t,x}_{r}\,dB_{r}.\nonumber
\end{eqnarray}
Notice that $Y^{i_{0},t,x}_{r}=u^{i_{0}}(r,X^{t,x}_{r})$ and we have that $|u^{i_{0}}(r,X^{t,x}_{r})|\leq C(1+|X^{t,x}_{r}|^q)$. Next let us set $\Sigma=1+\sup_{s\leq T}|X^{t,x}_{s}|$. Therefore $|Y^{i_{0},t,x}_{r}|\leq C_{q}\Sigma^{q}$ and $\left|g^{i_{0}}(X^{t,x}_{T})\right|\leq C_{q}\Sigma^{q}$.\\
By raising to the power $\frac{p}{2}$ and then taking expectation, it follows from (\ref{4.24bis}) and the fact of, there exists $C>0$ such that,\\ $\mathbb{E}\left[\left(\displaystyle\int^{T}_{s}\|U^{i_{0},t,x}_{r}\|^{2}_{\mathbb{L}^{2}(\lambda)}\,dr\right)^{\frac{p}{2}}\right]\leq C\mathbb{E}\left[\left(\sum_{s\leq r\leq T}(\Delta Y^{i_{0},t,x}_{r})^{2}\right)^{\frac{p}{2}}\right]$, (see \cite{len} p. 28-45),
\begin{eqnarray}\label{4.25ori}
& {}{} &\mathbb{E}\left[\left|Y^{i_{0},t,x}_{s}\right|^{p}\right]+\mathbb{E}\left[\left( \displaystyle\int^{T}_{s}\left|Z^{i_{0},t,x}_{r}\right|^{2}\,dr\right)^{\frac{p}{2}}\right] +\mathbb{E}\left[\left(\displaystyle\int^{T}_{s}\|U^{i_{0},t,x}_{r}\|^{2}_{\mathbb{L}^{2}(\lambda)}\,dr\right)^{\frac{p}{2}}\right]\nonumber\\
& {}{} &\leq 5^{\frac{p}{2}-1}\mathbb{E}\left[\left|g^{i_{0}}(X^{t,x}_{T})\right|^{p}\right]+5^{\frac{p}{2}-1}\mathbb{E}\left[\left(\displaystyle\int^{T}_{s}2\left|
Y^{i_{0},t,x}_{r}f^{(i_{0})}(r,X^{t,x}_{r},(u^{j}(X^{t,x}_{r}))_{j=1,m},Z^{i_{0},t,x}_{r},U^{i_{0},t,x}_{r})\right|\,dr\right)^{\frac{p}{2}}\right]\nonumber\\
& {}{} &+5^{\frac{p}{2}-1}\mathbb{E}\left[\left|\displaystyle\int^{T}_{s}Y^{i_{0},t,x}_{r}\,dK^{i_{0},t,x}_{r}\right|^{\frac{p}{2}}\right]\nonumber \\
& {}{} &+5^{\frac{p}{2}-1}\mathbb{E}\left[\left|\displaystyle\int^{T}_{s}
\displaystyle\int_{\mathrm{E}}2Y^{i_{0},t,x}_{r}U^{i_{0},t,x}_{r}\tilde{\mu}(\mathrm{d}r,\mathrm{d}e)\right|^{\frac{p}{2}}\right]+5^{\frac{p}{2}-1}\mathbb{E}\left[\left|\displaystyle\int^{T}_{s}2Y^{i_{0},t,x}_{r}Z^{i_{0},t,x}_{r}\,dB_{r}\right|^{\frac{p}{2}}\right].
\end{eqnarray}
For more comprehension,  we adopt the following scripture for inequality (\ref{4.25ori});\\ 
\begin{eqnarray}\label{4.26bis}
& {}{} &\mathbb{E}\left[\left|Y^{i_{0},t,x}_{s}\right|^{p}\right]+\mathbb{E}\left[\left( \displaystyle\int^{T}_{s}\left|Z^{i_{0},t,x}_{r}\right|^{2}\,dr\right)^{\frac{p}{2}}\right] +\mathbb{E}\left[\left(\displaystyle\int^{T}_{s}\|U^{i_{0},t,x}_{r}\|^{2}_{\mathbb{L}^{2}(\lambda)}\,dr\right)^{\frac{p}{2}}\right]\leq 5^{\frac{p}{2}-1}\mathbb{E}\left[\left|g^{i_{0}}(X^{t,x}_{T})\right|^{p}\right]\nonumber\\
& {}{} & +5^{\frac{p}{2}-1}T_{1}(s)+5^{\frac{p}{2}-1}T_{2}(s)+5^{\frac{p}{2}-1}T_{3}(s)+5^{\frac{p}{2}-1}T_{4}(s)
\end{eqnarray}
where,\\
\begin{eqnarray*}
T_{1}(s) & = & \mathbb{E}\left[\left(\displaystyle\int^{T}_{s}2\left|
Y^{i_{0},t,x}_{r}f^{(i_{0})}(r,X^{t,x}_{r},(u^{j}(X^{t,x}_{r}))_{j=1,m},Z^{i_{0},t,x}_{r},U^{i_{0},t,x}_{r})\right|\,dr\right)^{\frac{p}{2}}\right];\\ 
T_{2}(s) & = & \mathbb{E}\left[\left|\displaystyle\int^{T}_{s}Y^{i_{0},t,x}_{r}\,dK^{i_{0},t,x}_{r}\right|^{\frac{p}{2}}\right];\\ 
T_{3}(s) & = & \mathbb{E}\left[\left|\displaystyle\int^{T}_{s}
\displaystyle\int_{\mathrm{E}}2Y^{i_{0},t,x}_{r}U^{i_{0},t,x}_{r}\tilde{\mu}(\mathrm{d}r,\mathrm{d}e)\right|^{\frac{p}{2}}\right];\\ 
T_{4}(s) & = & \mathbb{E}\left[\left|\displaystyle\int^{T}_{s}2Y^{i_{0},t,x}_{r}Z^{i_{0},t,x}_{r}\,dB_{r}\right|^{\frac{p}{2}}\right].
\end{eqnarray*}
We will estimate $T_{1}(s)$, $T_{2}(s)$, $T_{3}(s)$ and $T_{4}(s)$, $\forall s\in[t,T]$.\\
(a) Before starting our estimations, let's linearize $f$ with respect to $(u^{j}(X^{t,x}_{r}))_{j=1,m}$ and $Z^{i_{0},t,x}_{r}$ i.e.\\
$f^{(i_{0})}(r,X^{t,x}_{r},(u^{j}(X^{t,x}_{r}))_{j=1,m},Z^{i_{0},t,x}_{r},U^{i_{0},t,x}_{r})=a^{t,x}_{r}Z^{i_{0},t,x}_{r}+b^{t,x}_{r}(u^{j}(X^{t,x}_{r}))_{j=1,m}+f^{(i_{0})}(r,X^{t,x}_{r},0,0,U^{i_{0},t,x}_{r})$;\\
where $a^{t,x}_{r}$ and $b^{t,x}_{r}$ are progressively measurable processes  respectively bounded by the Lipschitz constants of $f$ in $Z^{t,x}_{r}$ and $(u^{j}(X^{t,x}_{r}))_{j=1,m}$ i.e $|a^{t,x}_{r}|\leq C_{Z}$ and $|b^{t,x}_{r}|\leq\lambda_{1}$.\\
(b) We also take the fact that $f$ is Lipschitz in $U^{i_{0},t,x}_{r}$ i.e there exists a constant Lipschitz $\lambda_{2}$ such that $|f^{(i_{0})}(r,X^{t,x}_{r},0,0,U^{i_{0},t,x}_{r})|\leq|f^{(i_{0})}(r,X^{t,x}_{r},0,0,0)|+\lambda_{2}\|U^{i_{0},t,x}_{r}\|_{\mathbb{L}^{2}_{m}(\lambda)}$.\\
By combining (a) and (b) we have
\begin{eqnarray}\label{linea}
|f^{(i_{0})}(r,X^{t,x}_{r},(u^{j}(X^{t,x}_{r}))_{j=1,m},Z^{i_{0},t,x}_{r},U^{i_{0},t,x}_{r})|& \leq & |a^{t,x}_{r}Z^{i_{0},t,x}_{r}|+|b^{t,x}_{r}(u^{j}(X^{t,x}_{r}))_{j=1,m}|+|f^{(i_{0})}(r,X^{t,x}_{r},0,0,0)|\nonumber\\
& {}{} &+\lambda_{2}\|U^{i_{0},t,x}_{r}\|_{\mathbb{L}^{2}_{m}(\lambda)}.
\end{eqnarray}
Let's start our estimates.\\
\\
\underline{For $T_{1}(s)$}\\
By using (\ref{linea}), it follows that;
\begin{eqnarray}\label{4.25bis}
& {}{} &\displaystyle\int^{T}_{s}2\left|
Y^{i_{0},t,x}_{r}f^{(i_{0})}(r,X^{t,x}_{r},(u^{j}(X^{t,x}_{r}))_{j=1,m},Z^{i_{0},t,x}_{r},U^{i_{0},t,x}_{r})\right|\,dr\nonumber\\
& {}{} &\leq \displaystyle\int^{T}_{s}2\left|
Y^{i_{0},t,x}_{r}\left(a^{t,x}_{r}Z^{i_{0},t,x}_{r}\right)\right|\,dr +\displaystyle\int^{T}_{s}2\left|
Y^{i_{0},t,x}_{r}\left(b^{t,x}_{r}(u^{j}(X^{t,x}_{r}))_{j=1,m}\right)\right|\,dr+\lambda_{2}\displaystyle\int^{T}_{s}2\left|
Y^{i_{0},t,x}_{r}\left(U^{i_{0},t,x}_{r}\right)\right|dr\nonumber\\
& {}{} &+\displaystyle\int^{T}_{s}2\left|
Y^{i_{0},t,x}_{r}f^{(i_{0})}(r,X^{t,x}_{r},0,0,0)\right|\,dr\nonumber\\
& {}{} &\leq C^{2}_{q}C_{Z}T\epsilon^{-1}_{1}\Sigma^{2q}+\epsilon_{1}C_{Z}\displaystyle\int^{T}_{s}\left|Z^{i_{0},t,x}_{r}\right|^{2}\,dr+C^{2}_{q}\lambda_{1}T\epsilon^{-1}_{2}\Sigma^{2q}+\epsilon_{2}\lambda_{1}C^{2}_{q}T\Sigma^{2q}+C^{2}_{q}T\lambda_{2}\epsilon^{-1}_{3}\Sigma^{2q}\nonumber\\
& {}{} &+\lambda_{2}\epsilon_{3}\displaystyle\int^{T}_{s}\|U^{i_{0},t,x}_{r}\|^{2}_{\mathbb{L}^{2}_{m}(\lambda)}\,dr+C^{2}_{q}T\epsilon^{-1}_{4}\Sigma^{2q}+C^{2}_{q}T\epsilon_{4}\Sigma^{2q}.\nonumber\
\end{eqnarray}
By raising to the power $\frac{p}{2}$ and then taking expectation it follows that,\\
\newpage
\begin{eqnarray}\label{T1}
T_{1}(s) & \leq & 8^{\frac{p}{2}-1}CC^{p}_{q}\left\lbrace (C_{Z}T\epsilon^{-1}_{1})^{\frac{p}{2}}+(\lambda_{1}T\epsilon^{-1}_{2})^{\frac{p}{2}}+C^{p}_{q}(\epsilon_{2}\lambda_{1}T)^{\frac{p}{2}}+(T\lambda_{2}\epsilon^{-1}_{3})^{\frac{p}{2}}+(T\epsilon^{-1}_{4})^{\frac{p}{2}}\right.\nonumber\\
& {}{} & \left.+(T\epsilon_{4})^{\frac{p}{2}}\right\rbrace\left|1+|x|^{pq}\right|+8^{\frac{p}{2}-1}(\epsilon_{1}C_{Z})^{\frac{p}{2}}\mathbb{E}\left[\left(\displaystyle\int^{T}_{s}\left|Z^{i_{0},t,x}_{r}\right|^{2}\,dr\right)^{\frac{p}{2}}\right]\nonumber\\
& {}{} & +8^{\frac{p}{2}-1}(\lambda_{2}\epsilon_{3})^{\frac{p}{2}}\mathbb{E}\left[\left(\displaystyle\int^{T}_{s}\|U^{i_{0},t,x}_{r}\|^{2}_{\mathbb{L}^{2}_{m}(\lambda)}\,dr\right)^{\frac{p}{2}}\right].
\end{eqnarray} 
Before estimating $T_{2}(s)$, let us first give an estimate of $\mathbb{E}\left[|K^{i_{0},t,x}_{T}-K^{i_{0},t,x}_{s}|^{p}\right]$ which will serve us in that of $T_{2}(s)$.\\
\begin{eqnarray}
K^{i_{0},t,x}_{T}-K^{i_{0},t,x}_{s} &=& Y^{i_{0},t,x}_{s}-g^{i_{0}}(X^{t,x}_{T})-\displaystyle\int^{T}_{s}f^{(i_{0})}(r,X^{t,x}_{r},(u^{j}(X^{t,x}_{r}))_{j=1,m},Z^{i_{0},t,x}_{r},U^{i_{0},t,x}_{r})dr
\nonumber\\
& {}{} & +\displaystyle\int^{T}_{s}Z^{i_{0},t,x}_{r}\mathrm{d}
\mathrm{B}_{r}+\displaystyle\int^{T}_{s}
\displaystyle\int_{\mathrm{E}}U^{i_{0},t,x}_{r}(e)\tilde{\mu}(\mathrm{d}r,\mathrm{d}e).\nonumber
\end{eqnarray} 
By (\ref{linea}) and Cauchy-Schwartz inequality; it follows that,
\begin{eqnarray}
|K^{i_{0},t,x}_{T}-K^{i_{0},t,x}_{s}| & \leq & 2C_{q}\Sigma^{q}+C_{q}\lambda_{1}T\Sigma^{q}+C_{q}T\Sigma^{q}+C_{Z}\left(\displaystyle\int^{T}_{s}|Z^{i_{0},t,x}_{r}|^{2}\,dr\right)^{\frac{1}{2}}\nonumber\\
& {}{} &+\lambda_{2}\left(\displaystyle\int^{T}_{s}\|U^{i_{0},t,x}_{r}\|^{2}\,dr\right)^{\frac{1}{2}}+\displaystyle\int^{T}_{s}Z^{i_{0},t,x}_{r}\mathrm{d}
\mathrm{B}_{r}+\displaystyle\int^{T}_{s}
\displaystyle\int_{\mathrm{E}}U^{i_{0},t,x}_{r}(e)\tilde{\mu}(\mathrm{d}r,\mathrm{d}e).\nonumber
\end{eqnarray}
By raising to the power $p$, taking expectation and BDG inequality we have,
\begin{eqnarray}\label{KT}
\mathbb{E}\left[|K^{i_{0},t,x}_{T}-K^{i_{0},t,x}_{s}|^{p}\right]& \leq & 5^{p-1}CC^{p}_{q}\left\lbrace 2^{p}+(T\lambda_{1})^{p}+T^{p}\right\rbrace|1+|x|^{pq}|+5^{p-1}(C^{p}_{Z}+C_{p})\mathbb{E}\left(\displaystyle\int^{T}_{0}|Z^{i_{0},t,x}_{s}|^{2}\,ds\right)^{\frac{p}{2}}\nonumber\\
& {}{} &+5^{p-1}(\lambda^{p}_{2}+C_{p})\mathbb{E}\left\lbrace \displaystyle\int^{T}_{0} ds\|U^{i_{0},t,x}_{s}\|^{2}_{\mathbb{L}^{2}_{m}(\lambda)}\right\rbrace^{\frac{p}{2}}.
\end{eqnarray}

\underline{For $T_{2}(s)$}
\begin{eqnarray}\label{KT2}
\displaystyle\int^{T}_{s}|Y^{i_{0},t,x}_{s}|\mathrm{d}K^{i_{0},t,x}_{s} & \leq &  \displaystyle\int^{T}_{s}|(Y^{i_{0},t,x}_{s}-\ell(s,X^{t,x}_{s}))|\mathrm{d}K^{i_{0},t,x}_{s}+\displaystyle\int^{T}_{s}|\ell(s,X^{t,x}_{s})|\mathrm{d}K^{i_{0},t,x}_{s}\nonumber\\
{} & \leq & \sup_{s\leq T}|\ell(s,X^{t,x}_{s})|K^{i_{0},t,x}_{T}\nonumber\\
{} & \leq & \epsilon^{-1}_{5}\sup_{s\leq T}|\ell(s,X^{t,x}_{s})|^{2}+\epsilon_{5}(K^{i_{0},t,x}_{T})^{2}\nonumber\\
{} & \leq & \epsilon^{-1}_{5}C^{2}_{q}\Sigma^{2q}+\epsilon_{5}(K^{i_{0},t,x}_{T})^{2}.
\end{eqnarray} 
By using (\ref{KT}) and (\ref{KT2}); it follows that
\begin{eqnarray}
T_{2}(s) & \leq & 2^{\frac{p}{2}-1}CC^{p}_{q}(\epsilon^{-1}_{5})^{\frac{p}{2}}|1+|x|^{pq}|+2^{\frac{p}{2}-1}(\epsilon_{5})^{\frac{p}{2}}\mathbb{E}\left[(K^{i_{0},t,x}_{T})^{p}\right]\nonumber\\
{} & \leq & 2^{\frac{p}{2}-1}CC^{p}_{q}(\epsilon^{-1}_{5})^{\frac{p}{2}}|1+|x|^{pq}|+(\epsilon_{5})^{\frac{p}{2}}2^{\frac{p}{2}-1}7^{p-1}CC^{p}_{q}\left\lbrace 2^{p}+(T\lambda_{1})^{p}+T^{p}\right\rbrace|1+|x|^{pq}|+\nonumber\\
& {}{} &(\epsilon_{5})^{\frac{p}{2}}2^{\frac{p}{2}-1}7^{p-1}(C^{p}_{Z}+C_{p})\mathbb{E}\left(\displaystyle\int^{T}_{0}|Z^{i_{0},t,x}_{s}|^{2}\,ds\right)^{\frac{p}{2}}\nonumber\\
& {}{} &+(\epsilon_{5})^{\frac{p}{2}}2^{\frac{p}{2}-1}7^{p-1}(\lambda^{p}_{2}+C_{p})\mathbb{E}\left\lbrace \displaystyle\int^{T}_{0} ds\|U^{i_{0},t,x}_{s}\|^{2}_{\mathbb{L}^{2}_{m}(\lambda)}\right\rbrace^{\frac{p}{2}}\nonumber\\
T_{2}(s) & \leq & \left\lbrace 2^{\frac{p}{2}-1}CC^{p}_{q}(\epsilon^{-1}_{5})^{\frac{p}{2}}+(\epsilon_{5})^{\frac{p}{2}}2^{\frac{p}{2}-1}7^{p-1}CC^{p}_{q}(2^{p}+(T\lambda_{1})^{p}+T^{p})\right\rbrace \left(1+|x|^{pq}\right)\nonumber\\
& {}{} &+(\epsilon_{5})^{\frac{p}{2}}2^{\frac{p}{2}-1}7^{p-1}(C^{p}_{Z}+C_{p})\mathbb{E}\left(\displaystyle\int^{T}_{0}|Z^{i_{0},t,x}_{s}|^{2}\,ds\right)^{\frac{p}{2}}\nonumber\\
& {}{} &+(\epsilon_{5})^{\frac{p}{2}}2^{\frac{p}{2}-1}7^{p-1}(\lambda^{p}_{2}+C_{p})\mathbb{E}\left\lbrace \displaystyle\int^{T}_{0} ds\|U^{i_{0},t,x}_{s}\|^{2}_{\mathbb{L}^{2}_{m}(\lambda)}\right\rbrace^{\frac{p}{2}}.\nonumber
\end{eqnarray} 
\underline{For $T_{3}(s)$}\\
By BDG inequality,
\begin{eqnarray}
T_{3}(s) & \leq &  C_{p}\mathbb{E}\left(\displaystyle\int^{T}_{0}|Y^{i_{0},t,x}_{s}|^{2}|Z^{i_{0},t,x}_{s}|^{2}\,ds\right)^{\frac{p}{4}}\nonumber\\
{} & \leq & C_{p}\mathbb{E}\left(\sup_{s\leq T}|Y^{i_{0},t,x}_{s}|^{2}\displaystyle\int^{T}_{0}|Z^{i_{0},t,x}_{s}|^{2}\,ds\right)^{\frac{p}{4}}\nonumber\\
{} & \leq & C_{p}C^{p}_{q}\epsilon^{-1}_{6}(1+|x|^{pq})+C_{p}\epsilon_{6}\mathbb{E}\left(\displaystyle\int^{T}_{0}|Z^{i_{0},t,x}_{s}|^{2}\,ds\right)^{\frac{p}{2}}.
\end{eqnarray}
\underline{For $T_{4}(s)$}\\
By BDG inequality,
\begin{eqnarray}
T_{4}(s) & \leq &  C_{p}\mathbb{E}\left(\displaystyle\int^{T}_{0}|Y^{i_{0},t,x}_{s}|^{2}\|U^{i_{0},t,x}_{s}\|^{2}_{\mathbb{L}^{2}_{m}(\lambda)}\,ds\right)^{\frac{p}{4}}\nonumber\\
{} & \leq & C_{p}\mathbb{E}\left(\sup_{s\leq T}|Y^{i_{0},t,x}_{s}|^{2}\displaystyle\int^{T}_{0}\|U^{i_{0},t,x}_{s}\|^{2}_{\mathbb{L}^{2}_{m}(\lambda)}\,ds\right)^{\frac{p}{4}}\nonumber\\
{} & \leq & C_{p}C^{p}_{q}\epsilon^{-1}_{7}(1+|x|^{pq})+C_{p}\epsilon_{7}\mathbb{E}\left(\displaystyle\int^{T}_{0}\|U^{i_{0},t,x}_{s}\|^{2}_{\mathbb{L}^{2}_{m}(\lambda)}\,ds\right)^{\frac{p}{2}}.
\end{eqnarray}
Finally by taking estimation of $T_{1}(s)$, $T_{2}(s)$, $T_{3}(s)$, $T_{4}(s)$ and choosing $\epsilon_{1}$, $\epsilon_{2}$, $\epsilon_{3}$, $\epsilon_{4}$, $\epsilon_{5}$, $\epsilon_{6}$, $\epsilon_{7}$ such that;\\ 
$\{(\epsilon_{5})^{\frac{p}{2}}2^{\frac{p}{2}-1}7^{p-1}(C^{p}_{Z}+C_{p})+C_{p}\epsilon_{6}\}<1$,\\
$\{(\epsilon_{5})^{\frac{p}{2}}2^{\frac{p}{2}-1}7^{p-1}(\lambda^{p}_{2}+C_{p})+C_{p}\epsilon_{7}\}<1$,\\
and the sum of all coefficients of $(1+|x|^{pq})$ was small than $1$.\\ 
It follows then
\begin{equation}
\mathbb{E}\left[\left\lbrace \displaystyle\int^{T}_{0} ds\|U^{i_{0},t,x}_{s}\|^{2}_{\mathbb{L}^{2}_{m}(\lambda)}\right\rbrace^{\frac{p}{2}}\right]\leq C\left(1+\left|x\right|^{\rho}\right).
\end{equation} 
Where $\rho=pq$.\\
Finally since $i_{0}\in\{1,\ldots,m\}$ is arbitrary we then obtain the estimate (\ref{4.22}).
\end{proof}
\begin{proposition}
For any $i=1,\ldots,m$, $u^{i}$ belongs to $\mathcal{U}$.
\end{proposition}
\begin{proof}
Let $x$ and $x'$ be elements of $\mathbb{R}^{k}$. Let $(\vec{Y}^{t,x},Z^{t,x},U^{t,x},K^{t,x})$ $(\textrm{resp. }(\vec{Y}^{t,x'},Z^{t,x'},U^{t,x'},K^{t,x'}))$ be the solution of the RBSDE with jumps (\ref{3.17}) associated with $f(s,X^{t,x}_{s},y,\eta,\zeta,g(X^{t,x}_{T}))$\\
$(\textrm{resp. }f(s,X^{t,x'}_{s},y,\eta,\zeta,g(X^{t,x'}_{T})))$. Applying It\^o formula to $\left|\vec{Y}^{t,x}-\vec{Y}^{t,x'}\right|^{2}$ between $s$ and $T$, we have
\begin{eqnarray}\label{4.31}
& {}{} &\left|\vec{Y}^{t,x}_{s}-\vec{Y}^{t,x'}_{s}\right|^{2}+\displaystyle\int^{T}_{s}\left|\Delta Z_{r}\right|^{2}\,dr+\sum_{s\leq r\leq T}(\Delta_{r}\vec{Y}^{t,x}_{r})^{2}\\
& {}{} &=\left|g(X^{t,x}_{T})-g(X^{t,x'}_{T})\right|^{2}+2\displaystyle\int^{T}_{s}
<\left(\vec{Y}^{t,x}_{s}-\vec{Y}^{t,x'}_{s}\right),\Delta f(r)>\,dr+2\displaystyle\int^{T}_{s}\left(\vec{Y}^{t,x}_{r}
-\vec{Y}^{t,x'}_{r}\right)\,
d\left(\Delta K_{r}\right)\nonumber \\
& {}{} &-2\displaystyle\int^{T}_{s}
\displaystyle\int_{\mathrm{E}}\left(\vec{Y}^{t,x}_{r}
-\vec{Y}^{t,x'}_{r}\right)\left(\Delta U_{r}(e)\right)\tilde{\mu}(\mathrm{d}r,\mathrm{d}e)-2\displaystyle\int^{T}_{s} \left(\vec{Y}^{t,x}_{r}-\vec{Y}^{t,x'}_{r}\right)\left(\Delta Z_{r}\right)\,dB_{r};\nonumber
\end{eqnarray}
and taking expectation we obtain: $\forall s\in[t,T]$,
\begin{eqnarray}\label{4.33}
& {}{} &\mathbb{E}\left[\left|\vec{Y}^{t,x}_{s}-\vec{Y}^{t,x'}_{s}\right|^{2}+\displaystyle\int^{T}_{s}\left|\Delta Z_{r}\right|^{2}\,dr+\displaystyle\int^{T}_{s}\|\Delta U_{r}\|^{2}_{\mathbb{L}^{2}(\lambda)}\,dr\right]\\
& {}{} &\leq\mathbb{E}\left[\left|g(X^{t,x}_{T})-g(X^{t,x'}_{T})\right|^{2}+2\displaystyle\int^{T}_{s}
<\left(\vec{Y}^{t,x}_{s}-\vec{Y}^{t,x'}_{s}\right),\Delta f(r)>\,dr\right]\nonumber\\
& {}{} &\qquad\qquad+\mathbb{E}\left[2\displaystyle\int^{T}_{s}\left(\vec{Y}^{t,x}_{r}
-\vec{Y}^{t,x'}_{r}\right)\,
d\left(\Delta K_{r}\right)\right],\nonumber 
\end{eqnarray}
where the processes $\Delta X_{r}$, $\Delta Y_{r}$, $\Delta f(r)$, $\Delta K_{r}$, $\Delta Z_{r}$, $\Delta U_{r}$ and $\Delta \ell_{r}$ are defined as follows: $\forall r\in[t,T]$,\\
$\Delta f(r):=((\Delta f^{(i)}(r))_{i=1,m}=(f^{(i)}(r,X^{i;t,x}_{r},\vec{Y}^{t,x}_{r},Z^{i;t,x}_{r},U^{i;t,x}_{r})-f^{(i)}(r,X^{i;t,x'}_{r},\vec{Y}^{i;t,x'}_{r},Z^{i;t,x'}_{r},U^{i;t,x'}_{r}))_{i=1,m}$,
$\Delta X_{r}=X^{t,x}_{r}-X^{t,x'}_{r}$, $\Delta Y(r)=\vec{Y}^{t,x}_{r}-\vec{Y}^{t,x'}_{r}=(Y^{j;t,x}_{r}-Y^{j;t,x'}_{r})_{j=1,m}$,\\
$\Delta K_{r}=K^{t,x}_{r}-K^{t,x'}_{r}$, $\Delta Z_{r}=Z^{t,x}_{r}-Z^{t,x'}_{r}$, $\Delta U_{r}=U^{t,x}_{r}-U^{t,x'}_{r}$ and $\Delta\ell_{r}=\left(\ell(r,X^{t,x}_{r}
)-\ell(r,X^{t,x'}_{r})\right)$
($<\cdot,\cdot>$ is the usual scalar product on $\mathbb{R}^{m}$).
Now we will give an estimation of each three terms of the second member of inequality (\ref{4.33}).\\
$\bullet$ As for any $i\in\{1,\ldots,m\}$ $g^{i}$ belongs to $\mathcal{U}$; therefore\\
\begin{eqnarray}
\mathbb{E}\left[\left|g(X^{t,x}_{T})-g(X^{t,x'}_{T})\right|^{2}\right] & \leq & C\left|X^{t,x}_{T}-X^{t,x'}_{T}\right|^{2}(1+\left|X^{t,x}_{T}\right|^{2p}+\left|X^{t,x'}_{T}\right|^{2p})\nonumber\\
& \leq & \mathbb{E}\left[\left|x-x'\right|^{2}(1+\left|(X^{t,x}_{T}-x)+x\right|^{2p}+\left|(X^{t,x'}_{T}-x')+x'\right|^{2p}\right],\nonumber
\end{eqnarray}
and by subsequently using the triangle inequality, the relation of proposition $3.1$ and the fact that
 $$(a+b){^p}\leq 2^{p-1}(a^{p}+b^{p}).$$
\begin{equation}\label{4.34}
\mathbb{E}\left[\left|g(X^{t,x}_{T})-g(X^{t,x'}_{T})\right|^{2}\right]\leq C\left|x-x'\right|^{2}(1+\left|x\right|^{2p}+\left|x'\right|^{2p}),
\end{equation} 
$\bullet$ using (iii) of (\ref{3.17}): $\mathbb{E}\left[2\displaystyle\int^{T}_{s}\left(\vec{Y}^{t,x}_{r}
-\vec{Y}^{t,x'}_{r}\right)\,
d\left(\Delta K_{r}\right)\right]$ can be replaced by \\
$$\mathbb{E}\left[2\displaystyle\int^{T}_{s}\left(\ell(r,X^{t,x}_{r}
)-\ell(r,X^{t,x'}_{r})\right)\,
d\left(\Delta K_{r}\right)\right].$$\\
Now by  (\textbf{H1}) and Cauchy-Schwartz inequality we obtain:
\begin{equation}\label{4.35}
\mathbb{E}\left[\sup_{0\leq t\leq T}(\Delta\ell_{t})^{2}\right]\times\mathbb{E}\left[\left(\Delta K_{T}\right)^{2}\right]\leq 2CC'\left|x-x'\right|^{2}(1+\left|x\right|^{2p}+\left|x'\right|^{2p});
\end{equation}
where $C'=\mathbb{E}\left[\left(\Delta K_{T}\right)^{2}\right]$.\\  
$\bullet$ To complete our estimation of (\ref{4.33}) we need to deal with $\mathbb{E}\left[2\displaystyle\int^{T}_{s}
<\left(\vec{Y}^{t,x}_{s}-\vec{Y}^{t,x'}_{s}\right),\Delta f(r)>\,dr\right].$
Taking into account the expression of $f^{(i)}$ given by (\ref{2.13}) we then split $\Delta f(r)$ in the follows way: for $r\leq T$,
$$\Delta f(r)=(\Delta f(r))_{i=1,m}=\Delta_{1}(r)+\Delta_{2}(r)+\Delta_{3}(r)+\Delta_{4}(r)=(\Delta^{i}_{1}(r)+\Delta^{i}_{2}(r)+\Delta^{i}_{3}(r)+\Delta^{i}_{4}(r))_{i=1,m},$$
where  for any $i=1,\ldots,m$, 
\begin{eqnarray*}
\Delta^{i}_{1}(r) & = & h^{(i)}\left(r,X^{t,x}_{r},\vec{Y}^{t,x}_{r},Z^{i;t,x}_{r},\displaystyle\int_{E}\gamma^{i}(r,X^{t,x}_{r},e)U^{i;t,x}_{r}(e)\lambda(de)\right)\\
&{}{}& -h^{(i)}\left(r,X^{t,x'}_{r},\vec{Y}^{t,x}_{r},Z^{i;t,x}_{r},\displaystyle\int_{E}\gamma^{i}(r,X^{t,x}_{r},e)U^{i;t,x}_{r}(e)\lambda(de)\right);\\ 
\Delta^{i}_{2}(r) & = & h^{(i)}\left(r,X^{t,x'}_{r},\vec{Y}^{t,x}_{r},Z^{i;t,x}_{r},\displaystyle\int_{E}\gamma^{i}(r,X^{t,x}_{r},e)U^{i;t,x}_{r}(e)\lambda(de)\right)\\
&{}{}& -h^{(i)}\left(r,X^{t,x'}_{r},\vec{Y}^{t,x'}_{r},Z^{i;t,x}_{r},\displaystyle\int_{E}\gamma^{i}(r,X^{t,x}_{r},e)U^{i;t,x}_{r}(e)\lambda(de)\right);\\ 
\Delta^{i}_{3}(r) & = & h^{(i)}\left(r,X^{t,x'}_{r},\vec{Y}^{t,x'}_{r},Z^{i;t,x}_{r},\displaystyle\int_{E}\gamma^{i}(r,X^{t,x}_{r},e)U^{i;t,x}_{r}(e)\lambda(de)\right)\\
&{}{}& -h^{(i)}\left(r,X^{t,x'}_{r},\vec{Y}^{t,x'}_{r},Z^{i;t,x'}_{r},\displaystyle\int_{E}\gamma^{i}(r,X^{t,x}_{r},e)U^{i;t,x}_{r}(e)\lambda(de)\right);\\ 
\Delta^{i}_{4}(r) & = & h^{(i)}\left(r,X^{t,x'}_{r},\vec{Y}^{t,x'}_{r},Z^{i;t,x'}_{r},\displaystyle\int_{E}\gamma^{i}(r,X^{t,x}_{r},e)U^{i;t,x}_{r}(e)\lambda(de)\right)\\
&{}{}& -h^{(i)}\left(r,X^{t,x'}_{r},\vec{Y}^{t,x'}_{r},Z^{i;t,x'}_{r},\displaystyle\int_{E}\gamma^{i}(r,X^{t,x'}_{r},e)U^{i;t,x'}_{r}(e)\lambda(de)\right).
\end{eqnarray*}
By Cauchy-Schwartz inequality, the inequality $2ab\leq\epsilon a^{2}+\frac{1}{\epsilon}b^{2}$, the relation (2.11) and the estimate (\ref{3.16}) we have: 
\begin{eqnarray}\label{4.36}
\mathbb{E}\left[2\displaystyle\int^{T}_{s}
\scriptstyle<\Delta Y(r),\Delta_{1}(r)>\,dr\right] & \leq & \mathbb{E}\left[\frac{1}{\epsilon}\int^{T}_{s}\scriptstyle{|\Delta Y(r)|^{2}\,dr+C^{2}\epsilon}\displaystyle\int^{T}_{s}\scriptstyle{|X^{t,x}_{r}-X^{t,x'}_{r}|^{2}(1+|X^{t,x}_{r}|^{p}+|X^{t,x'}_{r}|^{p})^{2}\,dr}\right]\nonumber\\
& \leq & \mathbb{E}\left[\frac{1}{\epsilon}\int^{T}_{s}|\Delta Y(r)|^{2}\,dr\right]+C^{2}\epsilon|x-x'|^{2}(1+|x|^{p}+|x'|^{p})^{2}.
\end{eqnarray}
Besides since $h^{(i)}$ is Lipschitz w.r.t. $(y,z,q)$ then,
\begin{equation}\label{4.37}
\mathbb{E}\left[2\displaystyle\int^{T}_{s}<\Delta Y(r),\Delta_{2}(r)>\,dr\right]\leq 2C\mathbb{E}\left[\int^{T}_{s}|\Delta Y(r)|^{2}\,dr\right],
\end{equation}
and
\begin{equation}\label{4.38}
\mathbb{E}\left[2\displaystyle\int^{T}_{s}<\Delta Y(r),\Delta_{3}(r)>\,dr\right]\leq\mathbb{E}\left[\frac{1}{\epsilon}\int^{T}_{s}|\Delta Y(r)|^{2}\,dr+C^{2}\epsilon\int^{T}_{s}|\Delta Z(r)|^{2}\,dr\right].
\end{equation}
It remains to obtain a control of the last term. But for any $s\in[t,T]$ we have,
\begin{eqnarray}\label{4.39}
& {}{}& \mathbb{E}\left[2\displaystyle\int^{T}_{s}<\Delta Y(r),\Delta_{4}(r)>\,dr\right]\\
& \leq & 2C\mathbb{E}\left[\int^{T}_{s}|\Delta Y(r)|\,dr\times \left|\int_{E}\left(\gamma(r,X^{t,x}_{r},e)U^{t,x}_{r}(e)-\gamma(r,X^{t,x'}_{r},e)U^{t,x'}_{r}(e)\right)\,\lambda(de)\right|\right]\nonumber.
\end{eqnarray}
Next by splitting the crossing terms as follows
$\gamma(r,X^{t,x}_{r},e)U^{t,x}_{r}(e)-\gamma(r,X^{t,x'}_{r},e)U^{t,x'}_{r}(e)=\Delta U_{s}(e)\gamma(s,X^{t,x}_{s},e)+U^{t,x'}_{s}\left(\gamma(s,X^{t,x}_{s},e)-\gamma(s,X^{t,x'}_{s},e)\right)$\\
and setting $\Delta \gamma_{s}(e):=\left(\gamma(s,X^{t,x}_{s},e)-\gamma(s,X^{t,x'}_{s},e)\right)$,\\
we obtain,
\begin{eqnarray}\label{4.40}
\mathbb{E}\left[2\displaystyle\int^{T}_{s}<\Delta Y(r),\Delta_{4}(r)>\,dr\right]& \leq & 2C\mathbb{E}\left[\int^{T}_{s}|\scriptstyle\Delta Y(r)|\times\left(\displaystyle\int_{E}\scriptstyle(|U^{t,x'}_{r}(e)\Delta\gamma_{r}(e)|+|\Delta U_{r}(e)\gamma(r,X^{t,x}_{r},e)|)\,\lambda(de)\right)\,dr\right]\nonumber\\
& \leq & \frac{2}{\epsilon}\mathbb{E}\left[\int^{T}_{s}|\Delta Y(r)|^{2}\,dr\right]+C^{2}\epsilon\mathbb{E}\left[\int^{T}_{s}\left(\int_{E}(|U^{t,x'}_{r}(e)\Delta\gamma_{r}(e)|\lambda(de)\right)^{2}\,dr\right]\nonumber\\
& {}{} &+C^{2}\epsilon\mathbb{E}\left[\int^{T}_{s}\left(\int_{E}(|\Delta U_{r}(e)\gamma(r,X^{t,x}_{r},e)|\lambda(de)\right)^{2}\,dr\right].
\end{eqnarray}
By Cauchy-Schwartz inequality, (\ref{2.12}) and (\ref{3.16}), and the result of Lemma 4.1 it holds: 
\begin{eqnarray}\label{4.41}
\mathbb{E}\left[\int^{T}_{s}\left(\int_{E}(|U^{t,x'}_{r}(e)\Delta\gamma_{r}(e)|\lambda(de)\right)^{2}\,dr\right] & \leq & \mathbb{E}\left[\int^{T}_{s}\,dr\left(\int_{E}|U^{t,x'}_{r}(e)|^{2}\lambda(de)\right)\left(\int_{E}|\Delta\gamma_{r}(e)|^{2}\lambda(de)\right)\right]\nonumber\\
&\leq &  C\mathbb{E}\left[\{\scriptstyle\sup_{r\in[t,T]}|X^{t,x}_{r}-X^{t,x'}_{r}|^{2}(1+\sup_{r\in[t,T]}|X^{t,x}_{r}|^{p}+|X^{t,x'}_{r}|^{p})^{2}\,dr\}\right]\nonumber\\
& {}{} &\times \mathbb{E}\left[\int^{T}_{s}\,dr\left(\int_{E}|U^{t,x'}_{r}(e)|^{2}\lambda(de)\right)\right]\nonumber\\
& \leq & C\sqrt{\mathbb{E}\left[\{\scriptstyle\sup_{r\in[t,T]}|X^{t,x}_{r}-X^{t,x'}_{r}|^{4}(1+\sup_{r\in[t,T]}|X^{t,x}_{r}|^{p}+|X^{t,x'}_{r}|^{p})^{4}\}\right]}\nonumber\\
& {}{} &\times\sqrt{\mathbb{E}\left[\left\lbrace \int^{T}_{s}\,dr\left(\int_{E}|U^{t,x'}_{r}(e)|^{2}\lambda(de)\right)\right\rbrace^{2}\right]}\nonumber\\
& \leq & C\left|x-x'\right|^{2}(1+\left|x\right|^{2p}+\left|x'\right|^{2p}).
\end{eqnarray}
For some exponent $p$. On the other hand using once more Cauchy-Schwartz inequality and (\ref{2.12})-(i) we get
\begin{eqnarray}\label{4.42}
\mathbb{E}\left[\int^{T}_{s}\left(\int_{E}(\scriptstyle|\Delta U_{r}(e)\gamma(r,X^{t,x}_{r},e)|\lambda(de)\right)^{2}\,dr\right] & \leq & \mathbb{E}\left[\int^{T}_{s}\,dr\left(\int_{E}(\scriptstyle|\Delta U_{r}(e)|^{2}\lambda(de)\right)\left(\int_{E}|\gamma(r,X^{t,x}_{r},e)|^{2}\lambda(de)\right)\right]\nonumber\\
& \leq & C\mathbb{E}\left[\int^{T}_{s}\,dr\left(\int_{E}(|\Delta U_{r}(e)|^{2}\lambda(de)\right)\right].
\end{eqnarray}
Taking now into account inequalities (\ref{4.36})-(\ref{4.42}) we obtain:
\begin{eqnarray*}\label{4.43}
& {}{} &\mathbb{E}\left[\left|\vec{Y}^{t,x}_{s}-\vec{Y}^{t,x'}_{s}\right|^{2}+\displaystyle\int^{T}_{s}\left|\Delta Z_{r}\right|^{2}\,dr+\displaystyle\int^{T}_{s}\|\Delta U_{r}\|^{2}_{\mathbb{L}^{2}(\lambda)}\,dr\right]\\
& {}{} &\leq\mathbb{E}\left[\left|g(X^{t,x}_{T})-g(X^{t,x'}_{T})\right|^{2}+2\displaystyle\int^{T}_{s}
<\left(\vec{Y}^{t,x}_{s}-\vec{Y}^{t,x'}_{s}\right),\Delta f(r)>\,dr\right]\\
& {}{} &\qquad\qquad+\mathbb{E}\left[2\displaystyle\int^{T}_{s}\left(\vec{Y}^{t,x}_{r}
-\vec{Y}^{t,x'}_{r}\right)\,
d\left(\Delta K_{r}\right)\right]\\
& \leq & \left|x-x'\right|^{2}(1+\left|x\right|^{2p}+\left|x'\right|^{2p})(C+2CC'+C^{2}\epsilon+C^{3}\epsilon)+\left(\frac{3}{\epsilon}+2C\right)\mathbb{E}\left[\int^{T}_{s}|\Delta Y(r)|^{2}\,dr\right]\\
&{}{}&+C^{2}\epsilon\mathbb{E}\left[\int^{T}_{s}|\Delta Z(r)|^{2}\,dr\right]+C^{3}\epsilon\mathbb{E}\left[\int^{T}_{s}\,dr\left(\int_{E}(|\Delta U_{r}(e)|^{2}\lambda(de)\right)\right].
\end{eqnarray*}
Choosing now $\epsilon$ small enough we deduce the existence of a constant $C\geq 0$ such that for any $s\in[t,T]$,\\
$\mathbb{E}\left[|\Delta Y(s)|^{2}\right]\leq C\left|x-x'\right|^{2}(1+\left|x\right|^{2p}+\left|x'\right|^{2p})+\mathbb{E}\left[\displaystyle\int^{T}_{s}|\Delta Y(r)|^{2}\,dr\right]$\\
and by Gronwall lemma this implies that for any $s\in[t,T]$,\\
$$\mathbb{E}\left[|\Delta Y(s)|^{2}\right]\leq C\left|x-x'\right|^{2}(1+\left|x\right|^{2p}+\left|x'\right|^{2p}).$$
Finally in taking $s=t$ and considering (\ref{3.18}) we obtain the desired result.
\end{proof}

\begin{remark}
This result give also estimate of $U$ where we use the function $h^{(i)}$ $\forall~i=1,\ldots,m$ contrary in estimate (\ref{4.22}). 
\end{remark}

\begin{corollary}
For $u^{i}\in\mathcal{U}$ $\forall i=1,\ldots,m$ $B_{i}u^{i}$ defined in (\ref{2.2}) is well posed since the functions $\beta$ and $(\gamma_{i})_{i=1,m}$ verify (\ref{2.7}) and (\ref{2.12}) respectively.
\end{corollary}

\begin{proof}
The main point to notice is that $\lambda$ integrates $(1\wedge |e|^{p})$ $\forall p\geq 2$.\\
We have that
\begin{eqnarray*}
|\mathrm{B}_{i}u^{i}(t,x)| & \leq & \displaystyle\int_{\mathrm{E}}|\gamma^{i}(t,x,e)|\times|(u^{i}(t,x+\beta(t,x,e))-u^{i}(t,x))|\,\lambda(de)\\
&{}{}&\leq\displaystyle\int_{\mathrm{E}}C(1\wedge|e|)|\beta(t,x,e)|(1+|x+\beta(t,x,e)|^{p}+|x|^{p})\,\lambda(de)\\
&{}{}&\leq C^{2}(1+|x|^{p}(1+2^{p-1}))\displaystyle\int_{\mathrm{E}}C(1\wedge|e|^{2})\,\lambda(de)+(2^{\frac{p-1}{p}}C^{\frac{p+2}{p}})^{p}\displaystyle\int_{\mathrm{E}}C(1\wedge|e|^{p})\,\lambda(de).
\end{eqnarray*}
Which finish the proof.
\end{proof}

Now by remark 3, the last estimate of $U$ confirm the following result;
\begin{proposition}
For any $i=1,\ldots,m$, $(t,x)\in[0,T]\times\mathbb{R}^{k}$,
\begin{equation}
U^{i;t,x}_{s}(e)=u^{i}(s,X^{t,x}_{s-}+\beta(s,X^{t,x}_{s-},e))-u^{i}(s,X^{t,x}_{s-}),~~d\mathbb{P}\otimes ds\otimes d\lambda-\text{a.e. on}~\Omega\times[t,T]\times E.
\end{equation}

\end{proposition}
\begin{proof}
First note that since the measure $\lambda$ is note finite, then we cannot use the same technique as in \cite{hamaMor} where the authors use the jumps of processes and (\ref{3.18}).\\
In our case $U^{i;t,x}$ is only square integrable and not necessarily integrable w.r.t. $d\mathbb{P}\otimes ds\otimes d\lambda$. Therefore we first begin by truncating the L\'evy measure as the same way in \cite{hama}.\\
\textbf{Step 1: Truncation of the L\'evy measure}\\
For any $k\geq 1$, let us first introduce a new Poisson random measure $\mu_{k}$ (obtained from the truncation of $\mu$) and its associated compensator $\nu_{k}$ as follows:
$$\mu_{k}(ds,de)=1_{\{|e|\geq\frac{1}{k}\}}\mu(ds,de)~~\text{and }\nu_{k}(ds,de)=1_{\{|e|\geq\frac{1}{k}\}}\nu(ds,de).$$ 
Which means that, as usual, $\tilde{\mu_{k}}(ds,de):=(\mu_{k}-\nu_{k})(ds,de)$, is the associated random martingale measure.\\
The main point to notice is that 
\begin{eqnarray}
\lambda_{k}(E)=\displaystyle\int_{E}\,\lambda_{k}(de)& = &\displaystyle\int_{E}1_{\{|e|\geq\frac{1}{k}\}}\,\lambda(de)\nonumber\\
{}{}&=&\displaystyle\int_{\{|e|\geq\frac{1}{k}\}}\,\lambda(de)\nonumber\\
{}{}&=&\lambda(\{|e|\geq\frac{1}{k}\})<\infty.
\end{eqnarray}
As in \cite{hama}, let us introduce the process $^{k}X^{t,x}$ solving the following standard SDE of jump-diffusion type:
\begin{eqnarray}
& {}{} & ^{k}X^{t,x}_{s}=x+\displaystyle\int^{s}_{t}b(r,^{k}X^{t,x}_{r})\,dr+\displaystyle\int^{s}_{t}\sigma(r,^{k}X^{t,x}_{r})\,dB_{r}\nonumber\\
& {}{} &\qquad\qquad+\displaystyle\int^{s}_{t}\displaystyle\int_{\mathrm{E}}\beta(r,^{k}X^{t,x}_{r-},e)\tilde{\mu}_{k}\,(dr,de),~~~t\leq s\leq T;~^{k}X^{t,x}_{r}=x~\text{if }s\leq t.\nonumber\\
\end{eqnarray}
 Note that thanks to the assumptions on $b$, $\sigma$, $\beta$ the process $^{k}X^{t,x}$ exists and is unique. Moreover it satisfies the same estimates as in (\ref{3.16}) since $\lambda_{k}$ is just a truncation at the origin of $\lambda$ which integrates $(1\wedge|e|^{2})_{e\in E}$.\\
On the other hand let us consider the following Markovian RBSDE with jumps
 \begin{equation}\label{4.46}
\left
\{\begin{array}{ll}
(i)~\mathbb{E}\left[\sup_{s\leq T}\left|^{k}Y^{t,x}_{s}\right|^{2}+\displaystyle\int^{T}_{s}\left|^{k}Z^{t,x}_{r}\right|^{2}\,dr+\displaystyle\int^{T}_{s}\|^{k}U^{t,x}_{r}\|^{2}_{\mathbb{L}^{2}(\lambda_{k})}\,dr\right]<\infty\\
(ii)~^{k}{Y}^{t,x}:=(^{k}Y^{i,t,x})_{i=1,m}\in\mathcal{S}^{2}(\mathbb{R}^{m}),~^{k}Z^{t,x}:=(^{k}Z^{i,t,x})_{i=1,m}\in\mathbb{H}^{2}(\mathbb{R}^{m\times d}),\\
 ^{k}K^{t,x}:=(^{k}K^{i,t,x})_{i=1,m}\in\mathcal{A}^{2}_{c},~ ^{k}U^{t,x}:=(^{k}U^{i,t,x})_{i=1,m}\in\mathbb{H}^{2}(\mathbb{L}^{2}_{m}(\lambda_{k}));\\
(iii)~^{k}Y^{t,x}_{s}=g(^{k}X^{t,x}_{T})+\displaystyle\int^{T}_{s}f_{\mu_{k}}(r,^{k}X^{t,x}_{r},^{k}Y^{t,x}_{r},^{k}Z^{t,x}_{r},^{k}U^{t,x}_{r})\,dr+
^{k}K^{t,x}_{T}-^{k}K^{t,x}_{s}\\
\quad\quad\quad\quad\quad\quad
\quad\quad-\displaystyle\int^{T}_{s}\left\lbrace ^{k}Z^{t,x}_{r}\,d
B_{r}+\displaystyle\int_{\mathrm{E}}^{k}U^{t,x}_{r}(e)\tilde{\mu}_{k}(dr,de)\right\rbrace ,\quad s\leq T;\\
(iv)~^{k}Y^{i;t,x}_{s}\geq \ell(s,^{k}X^{t,x}_{s})~\textrm{and}~ \displaystyle\int^{T}_{0}(^{k}Y^{i;t,x}_{s}- \ell(s,^{k}X^{t,x}_{s}))\,d(^{k}K^{i;t,x}_{s})=0. 
\end{array}
\right.
\end{equation}
Finally let us introduce the following functions $(f^{(i)})_{i=1,m}$ defined by:
$\forall (t,x,y,z,\zeta)\in[0,T]\times\mathbb{R}^{k}\times\mathbb{R}^{m}\times\mathbb{R}^{m\times d}\times\mathbb{L}^{2}_{m}(\lambda_{k}),\\f_{\mu_{k}}(t,x,y,z,\zeta)=(f^{(i)}_{\mu_{k}}(t,x,y,z_{i},\zeta_{i}))_{i=1,m}:=\left( h^{(i)}\left(t,x,y,z,\displaystyle\int_{E}\gamma^{i}(t,x,e)\zeta_{i}(e)\lambda_{k}(de)\right)\right)_{i=1,m}$.
First let us emphasize that this latter RBSDE is related to the filtration $(\mathcal{F}^{k}_{s})_{s\leq T}$ generated by the Brownian motion and the independent random measure $\mu_{k}$. However this point does not raise major issues since for any $s\leq T$, $\mathcal{F}^{k}_{s}\subset \mathcal{F}_{s}$ and thanks to the relationship between $\mu$ and $\mu_{k}$.\\
Next by the properties of the functions $b$, $\sigma$, $\beta$ and by the same opinions of proposition $3.2$ and proposition $3.3$, there exists an unique quadriple $(^{k}Y^{t,x},^{k}K^{t,x},^{k}Z^{t,x},^{k}U^{t,x})$ solving (\ref{4.46}) and there also exists a function $u^{k}$ from $[0,T]\times \mathbb{R}^{k}$ into $\mathbb{R}^{m}$ of $\Pi^{c}_{g}$ such that
\begin{equation}\label{4.47}
\forall s\in[t,T],~~^{k}Y^{t,x}:=u^{k}(s,^{k}X^{t,x}),~\mathbb{P}-a.s.
\end{equation}  
Moreover as in proposition $4.2$, there exists positive constants $C$ and $p$ wich do not depend on $k$ such that:
\begin{equation}\label{4.48}
\forall t,x,x',~~|u^{k}(t,x)-u^{k}(t,x')|\leq C\left|x-x'\right|(1+\left|x\right|^{p}+\left|x'\right|^{p}).
\end{equation}
Finally as $\lambda_{k}$ is finite then we have the following relationship between the process $^{k}U^{t,x}:=(^{k}U^{i;t,x})_{i=1,m}$ and the deterministics functions $u^{k}:=(u^{k}_{i})_{i=1,m}$ (see \cite{hamaMor}): $\forall i=1,\ldots,m$; 
$$^{k}U^{i;t,x}_{s}(e)=u^{k}_{i}(s,^{k}X^{t,x}_{s-}+\beta(s,^{k}X^{t,x}_{s-},e))-u^{k}_{i}(s,^{k}X^{t,x}_{s-}),~~d\mathbb{P}\otimes ds\otimes d\lambda_{k}-a.e.~\text{on }\Omega\times[t,T]\times E.$$
This is mainly due to the fact that $^{k}U^{t,x}$ belongs to $\mathbb{L}^{1}\cap\mathbb{L}^{2}(ds\otimes d\mathbb{P}\otimes d\lambda_{k})$ since $\lambda_{k}(E)<\infty$ and then we can split the stochastic integral w.r.t. $\tilde{\mu}_{k}$ in (\ref{4.46}). Therefore for all $i=1,\ldots,m$,
\begin{equation}\label{4.49}
^{k}U^{i;t,x}_{s}(e)1_{\{|e|\geq \frac{1}{k}\}}=(u^{k}_{i}(s,^{k}X^{t,x}_{s-}+\beta(s,^{k}X^{t,x}_{s-},e))-u^{k}_{i}(s,^{k}X^{t,x}_{s-}))1_{\{|e|\geq \frac{1}{k}\}},~~d\mathbb{P}\otimes ds\otimes d\lambda_{k}-a.e.~\text{on }\Omega\times[t,T]\times E.
\end{equation}
\end{proof}
\textbf{Step 2: Convergence of the auxiliary processes}\\
Let's now prove the following convergence result;
\begin{eqnarray}\label{4.51}
&{}{}&\mathbb{E}\left[\sup_{s\leq T}\left|Y^{t,x}_{s}-^{k}Y^{t,x}_{s}\right|^{2}+(K^{t,x}_{T}-^{k}K^{t,x}_{T})^{2}+\displaystyle\int^{T}_{0}\left|Z^{t,x}_{s}-^{k}Z^{t,x}_{s}\right|^{2}\,ds
\right.\nonumber\\
&{}{}&\left.
+\displaystyle\int^{T}_{0}\,ds\displaystyle\int_{E}\lambda(de)\left|U^{t,x}_{s}(e)-^{k}U^{t,x}_{s}(e)1_{\{|e|\geq \frac{1}{k}\}}\right|^{2}\right]\substack{\longrightarrow\\ k\longrightarrow+\infty}0;
\end{eqnarray}
where $(Y^{t,x},K^{t,x},Z^{t,x},U^{t,x})$ is solution of the RBSDE with jumps (\ref{3.17}).\\
First note that the following convergence result was established in \cite{hama} 
\begin{equation}\label{4.50}
\mathbb{E}\left[\sup_{s\leq T}\left|X^{t,x}_{s}-^{k}X^{t,x}_{s}\right|^{2}\right]\substack{\longrightarrow\\ k\longrightarrow+\infty}0.
\end{equation}
We now focus on (\ref{4.51}). Note that we can apply Ito's formula, even if the RBSDEs are related to filtrations and Poisson random measures which are not the same, since:\\
(i) $\mathcal{F}^{k}_{s}\subset\mathcal{F}_{s}$, $\forall s\leq T$;\\
(ii) for any $s\leq T$, $\displaystyle\int^{s}_{0}\displaystyle\int_{\mathrm{E}}^{k}U^{i;t,x}(e)\tilde{\mu}_{k}\,(dr,de)=\displaystyle\int^{s}_{0}\displaystyle\int_{\mathrm{E}}^{k}U^{i;t,x}(e)1_{\{|e|\geq\frac{1}{k}\}}\tilde{\mu}\,(dr,de)$ and then the first $(\mathcal{F}^{k}_{s})_{s\leq T}-$martingale is also an $(\mathcal{F}_{s})_{s\leq T}-$martingale.
$\forall s\in[0,T]$,
\begin{eqnarray}
& {}{} &\left|\vec{Y}^{t,x}_{s}-^{k}Y^{t,x}_{s}\right|^{2}+\displaystyle\int^{T}_{0}\left|Z^{t,x}_{s}-^{k}Z^{t,x}_{s}\right|^{2},ds+\sum_{s\leq r\leq T}(^{k}\Delta_{r}\vec{Y}^{t,x}_{r})^{2}\nonumber\\
& {}{} &=\left|g(X^{t,x}_{T})-g(^{k}X^{t,x}_{T})\right|^{2}+2\displaystyle\int^{T}_{s}
\left(\vec{Y}^{t,x}_{r}-^{k}Y^{t,x}_{r}\right)\times ^{k}\Delta f(r)\,dr+2\displaystyle\int^{T}_{s}\left(\vec{Y}^{t,x}_{r}
-^{k}Y^{t,x}_{r}\right)\,
d\left(^{k}\Delta K_{r}\right)\nonumber \\
& {}{} &-2\displaystyle\int^{T}_{s}
\displaystyle\int_{\mathrm{E}}\left(\vec{Y}^{t,x}_{r}
-^{k}Y^{t,x}_{r}\right)\left(^{k}\Delta U_{r}(e)\right)\tilde{\mu}(\mathrm{d}r,\mathrm{d}e)-2\displaystyle\int^{T}_{s} \left(\vec{Y}^{t,x}_{r}-^{k}\vec{Y}^{t,x}_{r}\right)\left(^{k}\Delta Z_{r}\right)\,dB_{r};\nonumber
\end{eqnarray}
and taking expectation we obtain: $\forall s\in[t,T]$,

\begin{eqnarray}\label{4.53}
& {}{} &\mathbb{E}\left[\left|\vec{Y}^{t,x}_{s}-^{k}Y^{t,x}_{s}\right|^{2}+\left|^{k}\Delta K_{T}\right|^{2}+\displaystyle\int^{T}_{0}\left\lbrace\left|Z^{t,x}_{s}-^{k}Z^{t,x}_{s}\right|^{2}+\displaystyle\int_{E}\left|U^{t,x}_{s}-^{k}U^{t,x}_{s}1_{\{|e|\geq\frac{1}{k}\}}\right|^{2}\,\lambda(de)\right\rbrace\,ds\right]\nonumber\\
& {}{} &\leq\mathbb{E}\left[\left|g(X^{t,x}_{T})-g(^{k}X^{t,x}_{T})\right|^{2}+2\displaystyle\int^{T}_{s}
\left(\vec{Y}^{t,x}_{r}-^{k}Y^{t,x}_{r}\right)\times ^{k}\Delta f(r)\,dr\right]+\mathbb{E}\left[\sup_{s\leq T}\left|^{k}\Delta\ell_{s}\right|^{2}\right];\nonumber\\
\end{eqnarray}
where the processes $^{k}\Delta X_{r}$, $^{k}\Delta Y_{r}$, $^{k}\Delta f(r)$, $^{k}\Delta K_{r}$, $^{k}\Delta Z_{r}$, $^{k}\Delta U_{r}$ and $^{k}\Delta \ell_{r}$ are defined as follows: $\forall r\in[0,T]$,\\
$^{k}\Delta f(r):=((^{k}\Delta f^{(i)}(r))_{i=1,m}=(f^{(i)}(r,X^{t,x}_{r},\vec{Y}^{t,x}_{r},Z^{i;t,x}_{r},U^{i;t,x}_{r})-f^{(i)}_{k}(r,^{k}X^{t,x}_{r},^{k}Y^{t,x}_{r},^{k}Z^{t,x}_{r},^{k}U^{t,x}_{r}))_{i=1,m}$,
$^{k}\Delta X_{r}=X^{t,x}_{r}-^{k}X^{t,x}_{r}$, $^{k}\Delta Y(r)=\vec{Y}^{t,x}_{r}-^{k}Y^{t,x}_{r}=(Y^{j;t,x}_{r}-^{k}Y^{j;t,x}_{r})_{j=1,m}$,\\
$^{k}\Delta K_{r}=K^{t,x}_{r}-^{k}K^{t,x}_{r}$, $^{k}\Delta Z_{r}=Z^{t,x}_{r}-^{k}Z^{t,x}_{r}$, $^{k}\Delta U_{r}=U^{t,x}_{r}-^{k}U^{t,x}_{s}1_{\{|e|\geq\frac{1}{k}\}}$ and $^{k}\Delta\ell_{r}=\left(\ell(r,X^{t,x}_{r}
)-\ell(r,^{k}X^{t,x}_{r})\right)$.\\

Next let us set for $r\leq T$,
$$^{k}\Delta f(r)=(f(r,X^{t,x}_{r},\vec{Y}^{t,x}_{r},Z^{t,x}_{r},U^{t,x}_{r})-f_{k}(r,^{k}X^{t,x}_{r},^{k}Y^{t,x}_{r},^{k}Z^{t,x}_{r},^{k}U^{t,x}_{r}))=A(r)+B(r)+C(r)+D(r);$$
where  for any $i=1,\ldots,m$, 
\begin{eqnarray*}
A(r) & = & \left(h^{(i)}\left(r,X^{t,x}_{r},\vec{Y}^{t,x}_{r},Z^{i;t,x}_{r},\displaystyle\int_{E}\gamma^{i}(r,X^{t,x}_{r},e)U^{i;t,x}_{r}(e)\lambda(de)\right)
\right.\\
&{}{}&\left.
-h^{(i)}\left(r,^{k}X^{t,x}_{r},\vec{Y}^{t,x}_{r},Z^{i;t,x}_{r},\displaystyle\int_{E}\gamma^{i}(r,X^{t,x}_{r},e)U^{i;t,x}_{r}(e)\lambda(de)\right)\right)_{i=1,m};\\ 
B(r) & = & \left(h^{(i)}\left(r,^{k}X^{t,x}_{r},\vec{Y}^{t,x}_{r},Z^{i;t,x}_{r},\displaystyle\int_{E}\gamma^{i}(r,X^{t,x}_{r},e)U^{i;t,x}_{r}(e)\lambda(de)\right)
\right.\\
&{}{}&\left.
-h^{(i)}\left(r,^{k}X^{t,x}_{r},^{k}Y^{t,x}_{r},Z^{i;t,x}_{r},\displaystyle\int_{E}\gamma^{i}(r,X^{t,x}_{r},e)U^{i;t,x}_{r}(e)\lambda(de)\right)\right)_{i=1,m};\\ 
C(r) & = & \left(h^{(i)}\left(r,^{k}X^{t,x}_{r},^{k}Y^{t,x}_{r},Z^{i;t,x}_{r},\displaystyle\int_{E}\gamma^{i}(r,X^{t,x}_{r},e)U^{i;t,x}_{r}(e)\lambda(de)\right)
\right.\\
&{}{}&\left.
-h^{(i)}\left(r,^{k}X^{t,x}_{r},^{k}Y^{t,x}_{r},^{k}Z^{i;t,x}_{r},\displaystyle\int_{E}\gamma^{i}(r,X^{t,x}_{r},e)U^{i;t,x}_{r}(e)\lambda(de)\right)\right)_{i=1,m};\\ 
D(r) & = & \left( h^{(i)}\left(r,^{k}X^{t,x}_{r},^{k}Y^{t,x}_{r},^{k}Z^{i;t,x}_{r},\displaystyle\int_{E}\gamma^{i}(r,X^{t,x}_{r},e)U^{i;t,x}_{r}(e)\lambda(de)\right)
\right.\\
&{}{}&\left.
-h^{(i)}\left(r,^{k}X^{t,x}_{r},^{k}Y^{t,x}_{r},^{k}Z^{i;t,x}_{r},\displaystyle\int_{E}\gamma^{i}(r,^{k}X^{t,x}_{r},e)^{k}U^{i;t,x}_{r}(e)\lambda_{k}(de)\right)\right)_{i=1,m}.
\end{eqnarray*}
By (\ref{4.50}) and the of $g\in\mathcal{U}$ and $\ell\in\mathcal{U}$ we have,
\begin{equation}\label{4.54}
\mathbb{E}\left[\left|g(X^{t,x}_{T})-g(^{k}X^{t,x}_{T})\right|^{2}\right]\substack{\displaystyle\longrightarrow 0\\ k\rightarrow+\infty}
\end{equation}
and
\begin{equation}\label{4.55}
\mathbb{E}\left[\sup_{s\leq T}\left|\ell(X^{t,x}_{s})-\ell(^{k}X^{t,x}_{s})\right|^{2}\right]\substack{\displaystyle\longrightarrow 0\\ k\rightarrow+\infty}.
\end{equation}
Now we will interest to $\mathbb{E}\left[\displaystyle\int^{T}_{s}
\left(\vec{Y}^{t,x}_{r}-^{k}Y^{t,x}_{r}\right)\times ^{k}\Delta f(r)\,dr\right]$ for found (\ref{4.51}).\\
By (\ref{2.10}) and (\ref{2.11}), we have: $\forall r\in[0,T]$
\begin{eqnarray}\label{4.56} 
\left|A(r)\right| &\leq & C\left|X^{t,x}_{r}-^{k}X^{t,x}_{r}\right|(1+\left|X^{t,x}_{r}\right|^{p}+\left|^{k}X^{t,x}_{r}\right|^{p});\\
\left|B(r)\right| &\leq & C\left|\vec{Y}^{t,x}_{r}-^{k}Y^{t,x}_{r}\right|~\text{and } \left|C(r)\right|\leq \left|Z^{t,x}_{r}-^{k}Z^{t,x}_{r}\right|;\nonumber
\end{eqnarray}
where $C$ is a constant. Finally let us deal with $D(r)$ which is more involved. First note that $D(r)=(D_{i}(r))_{i=1,m}$ where 
\begin{eqnarray*}
D_{i}(r)& = & h^{(i)}\left(r,^{k}X^{t,x}_{r},^{k}Y^{t,x}_{r},^{k}Z^{i;t,x}_{r},\displaystyle\int_{E}\gamma^{i}(r,X^{t,x}_{r},e)U^{i;t,x}_{r}(e)\lambda(de)\right)\\
&{}{}&-h^{(i)}\left(r,^{k}X^{t,x}_{r},^{k}Y^{t,x}_{r},^{k}Z^{i;t,x}_{r},\displaystyle\int_{E}\gamma^{i}(r,^{k}X^{t,x}_{r},e)^{k}U^{i;t,x}_{r}(e)\lambda_{k}(de)\right).
\end{eqnarray*}
But as $h^{(i)}$ is Lipschitz w.r.t to the last component $q$ then,
\begin{eqnarray}\label{4.57}
\left|D(r)\right|^{2} & \leq & C\left\lbrace \displaystyle\int_{E}\left|\gamma^{i}(r,X^{t,x}_{r},e)U^{i;t,x}_{r}(e)-\gamma^{i}(r,^{k}X^{t,x}_{r},e)^{k}U^{i;t,x}_{r}(e)1_{\{|e|\geq\frac{1}{k}\}}\right|^{2}\lambda(de)\right\rbrace\nonumber\\
{}&\leq & C\left\lbrace \left\lbrace \displaystyle\int_{E}\left|\gamma^{i}(r,X^{t,x}_{r},e)-\gamma^{i}(r,^{k}X^{t,x}_{r},e)\right|\left|U^{i;t,x}_{r}(e)\right|\,\lambda(de)\right\rbrace^{2}
\right.\nonumber\\
&{}{}&\left.
+\left\lbrace \displaystyle\int_{E}\left|\gamma^{i}(r,X^{t,x}_{r},e)\right|\left| U^{i;t,x}_{r}(e)-^{k}U^{i;t,x}_{r}(e)1_{\{|e|\geq\frac{1}{k}\}}\right|\lambda(de)\right\rbrace^{2}\right\rbrace\nonumber\\
{}&\leq & C\left\lbrace \left|X^{t,x}_{r}-^{k}X^{t,x}_{r}\right|(1+\left|X^{t,x}_{r}\right|^{p}+\left|^{k}X^{t,x}_{r}\right|^{p})\left|\displaystyle\int_{E}U^{i;t,x}_{r}(e)\right|\lambda(de)\right\rbrace^{2}\nonumber\\
&{}{}&+C\displaystyle\int_{E}(1\wedge|e|)\left| U^{i;t,x}_{r}(e)-^{k}U^{i;t,x}_{r}(e)1_{\{|e|\geq\frac{1}{k}\}}\right|^{2}\lambda(de),
\end{eqnarray}
and (\ref{4.53}) become by using the majorations obtain in (\ref{4.56}) and in (\ref{4.57});
\begin{eqnarray}\label{4.58}
& {}{} &\mathbb{E}\left[\left|\vec{Y}^{t,x}_{s}-^{k}Y^{t,x}_{s}\right|^{2}+\left|^{k}\Delta K_{T}\right|^{2}+\displaystyle\int^{T}_{0}\left\lbrace\left|Z^{t,x}_{s}-^{k}Z^{t,x}_{s}\right|^{2}+\displaystyle\int_{E}\left|U^{t,x}_{s}-^{k}U^{t,x}_{s}1_{\{|e|\geq\frac{1}{k}\}}\right|^{2}\,\lambda(de)\right\rbrace\,ds\right]\nonumber\\
& {}{} &\leq\mathbb{E}\left[\left|g(X^{t,x}_{T})-g(^{k}X^{t,x}_{T})\right|^{2}\right]+\mathbb{E}\left[\sup_{s\leq T}\left|\ell(X^{t,x}_{s})-\ell(^{k}X^{t,x}_{s})\right|^{2}\right] +C\mathbb{E}\left[\displaystyle\int^{T}_{s}
\left|\vec{Y}^{t,x}_{s}-^{k}Y^{t,x}_{s}\right|^{2}\right]\nonumber\\
& {}{} &+C\mathbb{E}\left[\displaystyle\int^{T}_{0}\left|X^{t,x}_{r}-^{k}X^{t,x}_{r}\right|^{2}(1+\left|X^{t,x}_{r}\right|^{p}+\left|^{k}X^{t,x}_{r}\right|^{p})^{2}\,dr\right]\nonumber\\
& {}{} &+C\mathbb{E}\left[\displaystyle\int^{T}_{0}\,dr\left\lbrace \left|X^{t,x}_{r}-^{k}X^{t,x}_{r}\right|(1+\left|X^{t,x}_{r}\right|^{p}+\left|^{k}X^{t,x}_{r}\right|^{p})\displaystyle\int_{E}U^{i;t,x}_{r}(e)\lambda(de)\right\rbrace^{2}\right].\nonumber\\
\end{eqnarray}
The two first terms converge to $0$ by (\ref{4.54}) and (\ref{4.55}).\\
For the fourth term we have:
\begin{eqnarray*}
&{}{}&\mathbb{E}\left[\displaystyle\int^{T}_{0}\left|X^{t,x}_{r}-^{k}X^{t,x}_{r}\right|^{2}(1+\left|X^{t,x}_{r}\right|^{p}+\left|^{k}X^{t,x}_{r}\right|^{p})^{2}\,dr\right]\\
{}&\leq & \mathbb{E}\left[\sup_{r\leq T}\left|X^{t,x}_{r}-^{k}X^{t,x}_{r}\right|^{2}\displaystyle\int^{T}_{0}(1+\left|X^{t,x}_{r}\right|^{p}+\left|^{k}X^{t,x}_{r}\right|^{p})^{2}\,dr\right]\nonumber\\
{}&\leq & \left\lbrace\mathbb{E}\left[\sup_{r\leq T}\left|X^{t,x}_{r}-^{k}X^{t,x}_{r}\right|^{2}\right]\right\rbrace^{\frac{1}{2}}\left\lbrace\mathbb{E}\left[\left(\displaystyle\int^{T}_{0}(1+\left|X^{t,x}_{r}\right|^{p}+\left|^{k}X^{t,x}_{r}\right|^{p})^{2}\left|X^{t,x}_{r}-^{k}X^{t,x}_{r}\right|\right)^{2}\,dr\right]\right\rbrace^{\frac{1}{2}}.
\end{eqnarray*}
The first factor in the right-hand side of this inequality goes to $0$ when $k\rightarrow \infty$ due to (\ref{4.50}) and the second factor is uniformly bounded by the uniform estimates (\ref{3.16}) of $X^{t,x}$ and $^{k}X^{t,x}$.\\
Note also the last term converge to $0$ when $k\rightarrow \infty$, it is a consequence of (\ref{4.50}), the fact that $^{k}X^{t,x}$ verifies estimates (\ref{3.16}) uniformly, the Cauchy-Schwartz inequality (used twice) and finally (\ref{4.22}) of lemma $4.1$. Then by Gronwall's lemma we deduce first that for any $s\leq T$,
\begin{equation}\label{4.59}
\mathbb{E}\left[\left|\vec{Y}^{t,x}_{s}-^{k}Y^{t,x}_{s}\right|^{2}\right]\substack{\displaystyle\longrightarrow 0\\ k\rightarrow+\infty}
\end{equation}
and in taking $s=t$ we obtain $u^{k}(t,x)\substack{\displaystyle\longrightarrow u(t,x)\\ k\rightarrow+\infty}$. As $(t,x)\in[0,T]\times \mathbb{R}^{k}$ is arbitrary then u$^{k}\substack{\displaystyle\longrightarrow u\\ k\rightarrow+\infty}$ pointwisely.\\
Next going  back to (\ref{4.58}) take the limit w.r.t $k$ and using the uniform polynomial growth of $u^{k}$ and the Lebesgue dominated convergence theorem as well, to obtain:
\begin{equation}\label{4.60}
\mathbb{E}\left[\displaystyle\int^{T}_{t}\displaystyle\int_{E}\left|U^{t,x}_{s}-^{k}U^{t,x}_{s}1_{\{|e|\geq\frac{1}{k}\}}\right|^{2}\,\lambda(de)\,ds\right]\substack{\displaystyle\longrightarrow 0\\ k\rightarrow+\infty}.
\end{equation}
\textbf{Step 3: Conclusion}\\
First note that by (\ref{4.48}) and the pointwise convergence of $(u^{k})_{k}$ to $u$, if $(x_{k})_{k}$ is a sequence of $\mathbb{R}^{k}$ which converge to $x$ then $((u^{k}(t,x_{k}))_{k})$ converge to $u(t,x)$.\\
Now let us consider a subsequence which we still denote by $\{k\}$ such that $\sup_{s\leq T}\left|X^{t,x}_{s}-^{k}X^{t,x}_{s}\right|^{2}\substack{\displaystyle\longrightarrow 0\\ k\rightarrow+\infty}$, $\mathbb{P}$-a.s. (and then $\left|X^{t,x}_{s-}-^{k}X^{t,x}_{s-}\right|\substack{\displaystyle\longrightarrow 0\\ k\rightarrow+\infty}$ since $\left|X^{t,x}_{s-}-^{k}X^{t,x}_{s-}\right|\leq \sup_{s\leq T}\left|X^{t,x}_{s}-^{k}X^{t,x}_{s}\right|^{2}$). By (\ref{4.50}), this subsequence exists. As the mapping $x\mapsto \beta(t,x,e)$ is Lipschitz then the sequence
\begin{eqnarray}\label{4.61}
&{}{}&\left(^{k}U^{t,x}_{s}(e)1_{\{|e|\geq\frac{1}{k}\}}\right)_{k}=\left((u^{k}_{i}(s,^{k}X^{t,x}_{s-}+\beta(s,^{k}X^{t,x}_{s-},e))-u^{k}_{i}(s,^{k}X^{t,x}_{s-}))1_{\{|e|\geq\frac{1}{k}\}}\right)_{k\geq 1}\substack{\displaystyle\longrightarrow {}\\ k\rightarrow+\infty}\nonumber\\
&{}{}& (u_{i}(s,X^{t,x}_{s-}+\beta(s,X^{t,x}_{s-},e))-u_{i}(s,X^{t,x}_{s-})),\quad d\mathbb{P}\otimes ds\otimes d\lambda-a.e.\quad \text{on }\Omega\times[t,T]\times E\quad
\end{eqnarray}
for any $i=1,\ldots,m$. Finally from (\ref{4.60}) we deduce that
\begin{equation}\label{4.62}
U^{t,x}_{s}(e)=(u_{i}(s,X^{t,x}_{s-}+\beta(s,X^{t,x}_{s-},e))-u_{i}(s,X^{t,x}_{s-})),\quad\text{on }\Omega\times[t,T]\times E
\end{equation}
which is the desired result.

\section{The main result}
First we give the definition of viscosity solution of IPDEs as given in \cite{hama} and \cite{hamaMor}. Our main result deal with this definition.
\begin{definition}
We say that a family of deterministics functions $u=(u^{i})_{i=1,m}$ which belongs to $\mathcal{U}\quad\forall i\in\{1,\ldots,m\}$ is a viscosity sub-solution (resp. super-solution) of the IPDE (\ref{eq1}) if:\\
$(i)\quad \forall x\in\mathbb{R}^{k}$, $u^{i}(x,T)\leq g^{i}(x)$ (resp. $u^{i}(x,T)\geq g^{i}(x)$);\\
$(ii)\quad\text{For any } (t,x)\in[0,T]\times\mathbb{R}^{k}$ and any function $\phi$ of class $C^{1,2}([0,T]\times\mathbb{R}^{k})$ such that $(t,x)$ is a global maximum  point of $u^{i}-\phi$ (resp. global minimum  point of $u^{i}-\phi$) and $(u^{i}-\phi)(t,x)=0$ one has
\begin{equation}\label{5.63}
\min\left\lbrace u^{i}(t,x)-\ell(t,x);-\partial_{t}\phi(t,x)-\mathcal{L}^{X}\phi(t,x)-h^{i}(t,x,(u^{j}(t,x))_{j=1,m},\sigma^{\top}(t,x))D_{x}\phi(t,x),B_{i}u^{i}(t,x))\right\rbrace\leq 0 
\end{equation}
$\left(resp.
\right.$
\begin{equation}\label{5.64}
\left.
\min\left\lbrace u^{i}(t,x)-\ell(t,x);-\partial_{t}\phi(t,x)-\mathcal{L}^{X}\phi(t,x)-h^{i}(t,x,(u^{j}(t,x))_{j=1,m},\sigma^{\top}(t,x))D_{x}\phi(t,x),B_{i}u^{i}(t,x))\right\rbrace\geq 0\right).
\end{equation}
The family $u=(u^{i})_{i=1,m}$ is a viscosity solution of (\ref{eq1}) if it is both a viscosity sub-solution and viscosity super-solution.\\
Note that $\mathcal{L}^{X}\phi(t,x)=b(t,x)^{\top}\mathrm{D}_{x}\phi(t,x)+\frac{1}{2}\mathrm{Tr}(\sigma\sigma^{\top}(t,x)\mathrm{D}^{2}_{xx}\phi(t,x))+\mathrm{K}\phi(t,x)$;\\
where $\mathrm{K}\phi(t,x)=\displaystyle\int_{\mathrm{E}}(\phi(t,x+\beta(t,x,e))-\phi(t,x)-\beta(t,x,e)^{\top}\mathrm{D}_{x}\phi(t,x))\lambda(de).$
\end{definition}
\begin{theorem}
Under assumptions (\textbf{H1}), (\textbf{H2}) and (\textbf{H3}), the IPDE (\ref{eq1}) has unique solution which is the $m$-tuple of functions $(u^{i})_{i=1,m}$ defined in proposition $3.3$ by (\ref{3.18}).  
\end{theorem}
\begin{proof}
\emph{\underline{Step $1$}:} \emph{Existence}\\
Assume that assumptions (\textbf{H1}), (\textbf{H2}) and (\textbf{H3}) are fulfilled, then the following multi-dimensional RBSDEs with jumps
\begin{equation}\label{5.65}
\left
\{\begin{array}{ll}
(i)~\underline{\vec{Y}}^{t,x}:=(\underline{Y}^{i;t,x})_{i=1,m}\in\mathcal{S}^{2}(\mathbb{R}^{m}),~\underline{Z}^{t,x}:=(\underline{Z}^{i;t,x})_{i=1,m}\in\mathbb{H}^{2}(\mathbb{R}^{m\times d}),~
 \underline{K}^{t,x}:=(\underline{K}^{i;t,x})_{i=1,m}\in\mathcal{A}^{2}_{c},\\\underline{U}^{t,x}:=(\underline{U}^{i;t,x})_{i=1,m}\in\mathbb{H}^{2}(\mathbb{L}^{2}_{m}(\lambda));\\
(ii)~\underline{Y}^{i;t,x}_{s}= g^{i}(X^{t,x}_{T})+
\underline{K}^{i;t,x}_{T}-\underline{K}^{i;t,x}_{s}-\displaystyle\int^{T}_{s}\underline{Z}^{i;t,x}\mathrm{d}
\mathrm{B}_{r}-\displaystyle\int^{T}_{s}
\displaystyle\int_{\mathrm{E}}\underline{U}^{i;t,x}_{r}(e)\tilde{\mu}(\mathrm{d}r,\mathrm{d}e).\\
\quad\quad\quad+\displaystyle\int^{T}_{s}h^{(i)}(r,X^{t,x}_{r},\underline{Y}^{i;t,x}_{r},\underline{Z}^{i;t,x}_{r},\displaystyle\int_{\mathrm{E}}\gamma^{i}(t,X^{t,x}_{r},e)\{(u^{i}(t,X^{t,x}_{r-}+\beta(t,X^{t,x}_{r-},e))-u^{i}(t,X^{t,x}_{r-}))\}\,\lambda(de))dr\\
(iii)~\underline{Y}^{i;t,x}_{s}\geq \ell(s,X^{t,x}_{s})~\textrm{and}~ \displaystyle\int^{T}_{0}(\underline{Y}^{i;t,x}_{s}- \ell(s,X^{t,x}_{s}))\mathrm{d}\underline{K}^{i;t,x}_{s}=0; 
\end{array}
\right.
\end{equation}
has unique solution  $(\underline{Y},\underline{Z},\underline{K},\underline{U})$.
Next as for any $i=1,\ldots,m$, $u^{i}$ belongs to $\mathcal{U}$, then by proposition $3.3$ the (\ref{3.18}), there exists a family of deterministics continuous functions of polynomial growth $(\underline{u}^{i})_{i=1,m}$ that fact for any $(t,x)\in[0,T]\times\mathbb{R}^{k}$,
$$\forall s\in[t,T],\qquad \underline{Y}^{i;t,x}_{s}=\underline{u}^{i}(s,X^{t,x}_{s}).$$
Such that by the same proposition, the family $(\underline{u}^{i})_{i=1,m}$ is a viscosity solution of the following system:
\begin{equation}\label{5.66}
\left
\{\begin{array}{ll}
\min\Big\{\underline{u}^{i}(t,x)-\ell(t,x);-\partial_{t}\underline{u}^{i}(t,x)-b(t,x)^{\top}\mathrm{D}_{x}\underline{u}^{i}(t,x)-\frac{1}{2}\mathrm{Tr}(\sigma\sigma^{\top}(t,x)\mathrm{D}^{2}_{xx}\underline{u}^{i}(t,x))\\
\quad\quad-\mathrm{K}_{i}\underline{u}^{i}(t,x)-\mathit{h}^{(i)}(t,x,(\underline{u}^{j}(t,x))_{j=1,m},(\sigma^{\top}\mathrm{D}_{x}\underline{u}^{i})(t,x),\mathrm{B}_{i}u^{i}(t,x))\Big\}=0,\quad (t,x)\in\left[ 0,T\right] \times\mathbb{R}^{k};\\
u^{i}(T,x)=g^{i}(x).
\end{array}
\right.
\end{equation}
Now we have the family $(\underline{u}^{i})_{i=1,m}$ is a viscosity solution, our main objective is to found relation between $(\underline{u}^{i})_{i=1,m}$ and $(u^{i})_{i=1,m}$ which is defined in (\ref{3.18}).\\
For this, let us consider the system of RBSDE with jumps
\begin{equation}\label{5.67}
\left
\{\begin{array}{ll}
(i)~\vec{Y}^{t,x}:=(Y^{i;t,x})_{i=1,m}\in\mathcal{S}^{2}(\mathbb{R}^{m}),~Z^{t,x}:=(Z^{i;t,x})_{i=1,m}\in\mathbb{H}^{2}(\mathbb{R}^{m\times d}),~K^{t,x}:=(K^{i;t,x})_{i=1,m}\in\mathcal{A}^{2}_{c},\\U^{t,x}:=(U^{i;t,x})_{i=1,m}\in\mathbb{H}^{2}(\mathbb{L}^{2}_{m}(\lambda));\\
(ii)~Y^{i;t,x}_{s}= g^{i}(X^{t,x}_{T})+
K^{i;t,x}_{T}-K^{i;t,x}_{s}-\displaystyle\int^{T}_{s}Z^{i;t,x}\mathrm{d}
\mathrm{B}_{r}-\displaystyle\int^{T}_{s}
\displaystyle\int_{\mathrm{E}}U^{i;t,x}_{r}(e)\tilde{\mu}(\mathrm{d}r,\mathrm{d}e).\\
\quad\quad\quad+\displaystyle\int^{T}_{s}h^{(i)}(r,X^{t,x}_{r},Y^{i;t,x}_{r},\underline{Z}^{i;t,x}_{r},\displaystyle\int_{\mathrm{E}}\gamma^{i}(t,X^{t,x}_{r},e)U^{i;t,x}_{r}(e)\,\lambda(de))dr;\\
(iii)~Y^{i;t,x}_{s}\geq \ell(s,X^{t,x}_{s})~\textrm{and}~ \displaystyle\int^{T}_{0}(Y^{i;t,x}_{s}- \ell(s,X^{t,x}_{s}))\mathrm{d}K^{i;t,x}_{s}=0. 
\end{array}
\right.
\end{equation}
By uniqueness of the solution of the RBSDEs with jumps (\ref{5.64}), that for any $s\in[t,T]$ and $\forall i\in\{1\ldots,m\}$, $\underline{Y}^{i;t,x}_{s}=Y^{i;t,x}_{s}$.\\
Therefore  $\underline{u}^{i}=u^{i}$, such that by (\ref{4.60}) we obtain $U^{t,x}_{s}(e)=(u_{i}(s,X^{t,x}_{s-}+\beta(s,X^{t,x}_{s-},e))-u_{i}(s,X^{t,x}_{s-})),\quad\text{on }\Omega\times[t,T]\times E$, which give the viscosity solution in the sense of definition $5.1$ (see \cite{hama}) by pluging (\ref{4.61}) in $h^{(i)}$ of (\ref{5.66}).

\emph{\underline{Step $2$}:} \emph{Uniqueness}\\

For uniqueness, let $(\overline{u}^{i})_{i=1,m}$ be another family of $\mathcal{U}$ which is solution viscosity of the system (\ref{eq1}) in the sense of definition $5.1$ and we consider RBSDE with jumps defined with $\overline{u}^{i}$.   
\begin{equation}\label{5.68}
\left
\{\begin{array}{ll}
(i)~\vec{\overline{Y}}^{t,x}:=(\overline{Y}^{i;t,x})_{i=1,m}\in\mathcal{S}^{2}(\mathbb{R}^{m}),~\overline{Z}^{t,x}:=(\overline{Z}^{i;t,x})_{i=1,m}\in\mathbb{H}^{2}(\mathbb{R}^{m\times d}),~\overline{K}^{t,x}:=(\overline{K}^{i;t,x})_{i=1,m}\in\mathcal{A}^{2}_{c},\\\overline{U}^{t,x}:=(\overline{U}^{i;t,x})_{i=1,m}\in\mathbb{H}^{2}(\mathbb{L}^{2}_{m}(\lambda));\\
(ii)~\overline{Y}^{i;t,x}_{s}= g^{i}(X^{t,x}_{T})+
\overline{K}^{i;t,x}_{T}-\overline{K}^{i;t,x}_{s}-\displaystyle\int^{T}_{s}\overline{Z}^{i;t,x}\mathrm{d}
\mathrm{B}_{r}-\displaystyle\int^{T}_{s}
\displaystyle\int_{\mathrm{E}}\overline{U}^{i;t,x}_{r}(e)\tilde{\mu}(\mathrm{d}r,\mathrm{d}e).\\
\quad\quad\quad+\displaystyle\int^{T}_{s}h^{(i)}(r,X^{t,x}_{r},\overline{Y}^{i;t,x}_{r},\overline{Z}^{i;t,x}_{r},\displaystyle\int_{\mathrm{E}}\gamma^{i}(t,X^{t,x}_{r},e)(\overline{u}_{i}(s,X^{t,x}_{s-}+\beta(s,X^{t,x}_{s-},e))-\overline{u}_{i}(s,X^{t,x}_{s-}))\,\lambda(de))dr;\\
(iii)~\overline{Y}^{i;t,x}_{s}\geq \ell(s,X^{t,x}_{s})~\textrm{and}~ \displaystyle\int^{T}_{0}(\overline{Y}^{i;t,x}_{s}- \ell(s,X^{t,x}_{s}))\mathrm{d}\overline{K}^{i;t,x}_{s}=0. 
\end{array}
\right.
\end{equation}
By Feynman Kac formula $\overline{u}^{i}(s,X^{t,x}_{s})=Y^{i;t,x}_{s}$ where $Y^{i;t,x}_{s}$ satisfies the RBSDE with jumps (\ref{eq2}) associated to IPDE (\ref{eq1}).\\
Since that the RBSDE with jumps (\ref{5.66}) has solution and it is unique by   
assumed that (\textbf{H1}), (\textbf{H2}) and (\textbf{H3}) are verified. By proposition $3.3$ the (\ref{3.18}), there exists a family of deterministic continuous functions of polynomial growth $(v^{i})_{i=1,m}$ that fact for any $(t,x)\in[0,T]\times\mathbb{R}^{k}$,
$$\forall s\in[t,T],\qquad \overline{Y}^{i;t,x}_{s}=v^{i}(s,X^{t,x}_{s}).$$
Such that by the same proposition, the family $(v^{i})_{i=1,m}$ is a viscosity solution of the following system: 
\begin{equation}\label{5.69}
\left
\{\begin{array}{ll}
\min\Big\{v^{i}(t,x)-\ell(t,x);-\partial_{t}v^{i}(t,x)-b(t,x)^{\top}\mathrm{D}_{x}v^{i}(t,x)-\frac{1}{2}\mathrm{Tr}(\sigma\sigma^{\top}(t,x)\mathrm{D}^{2}_{xx}v^{i}(t,x))\\
\quad\quad-\mathrm{K}_{i}v^{i}(t,x)-\mathit{h}^{(i)}(t,x,(v^{j}(t,x))_{j=1,m},(\sigma^{\top}\mathrm{D}_{x}v^{i})(t,x),\mathrm{B}_{i}\overline{u}^{i}(t,x))\Big\}=0,\quad (t,x)\in\left[ 0,T\right] \times\mathbb{R}^{k};\\
u^{i}(T,x)=g^{i}(x).
\end{array}
\right.
\end{equation}
By uniqueness of solution of (\ref{5.67}) $\overline{u}^{i}$ is viscosity solution of (\ref{5.68}); and by proposition $3.3$ $v^{i}=\overline{u}^{i}$ $\forall i\in\{1,\ldots,m\}$.\\
Now for completing our proof we show that on $\Omega\times[t,T]\times E$, $ds\otimes d\mathbb{P}\otimes d\lambda-\text{a.e.}\quad\forall i\in\{1,\ldots, m\}$;
\begin{eqnarray}\label{5.71}
\overline{U}^{i;t,x}_{s}(e) & = & (v^{i}(s,X^{t,x}_{s-}+\beta(s,X^{t,x}_{s-},e))-v^{i}(s,X^{t,x}_{s-}))\nonumber\\
{} & = & (\overline{u}_{i}(s,X^{t,x}_{s-}+\beta(s,X^{t,x}_{s-},e))-\overline{u}_{i}(s,X^{t,x}_{s-})).
\end{eqnarray} 
By Remark $3.4$ in \cite{hama}; let us considere $(x_{k})_{k\geq 1}$ a sequence of $\mathbb{R}^{k}$ which converges to $x\in\mathbb{R}^{k}$ and the two following RBSDE with jumps (adaptation is w.r.t. $\mathcal{F}^{k}$):
\begin{equation}\label{5.72}
\left
\{\begin{array}{ll}
(i)~\vec{\overline{Y}}^{k,t,x}:=(\overline{Y}^{i;k,t,x})_{i=1,m}\in\mathcal{S}^{2}(\mathbb{R}^{m}),~\overline{Z}^{k,t,x}:=(\overline{Z}^{i;k,t,x})_{i=1,m}\in\mathbb{H}^{2}(\mathbb{R}^{m\times d}),\\\overline{K}^{k,t,x}:=(\overline{K}^{i;k,t,x})_{i=1,m}\in\mathcal{A}^{2}_{c},~\overline{U}^{k,t,x}:=(\overline{U}^{i;k,t,x})_{i=1,m}\in\mathbb{H}^{2}(\mathbb{L}^{2}_{m}(\lambda));\\
(ii)~\overline{Y}^{i;k,t,x}_{s}= g^{i}(X^{k,t,x}_{T})+
\overline{K}^{i;k,t,x}_{T}-\overline{K}^{i;k,t,x}_{s}-\displaystyle\int^{T}_{s}\overline{Z}^{i;k,t,x}\mathrm{d}
\mathrm{B}_{r}-\displaystyle\int^{T}_{s}
\displaystyle\int_{\mathrm{E}}\overline{U}^{i;k,t,x}_{r}(e)\tilde{\mu}(\mathrm{d}r,\mathrm{d}e)\\
\quad\quad\quad+\displaystyle\int^{T}_{s}h^{(i)}\left(r,X^{k,t,x}_{r},\overline{Y}^{i;k,t,x}_{r},\overline{Z}^{i;k,t,x}_{r},
\right.\\

\left.
\qquad\qquad\qquad\qquad\qquad\displaystyle\int_{\mathrm{E}}\gamma^{i}(t,X^{k,t,x_{k}}_{r},e)(\overline{u}_{i}(s,X^{k,t,x_{k}}_{s-}+\beta(s,X^{k,t,x_{k}}_{s-},e))-\overline{u}_{i}(s,X^{k,t,x_{k}}_{s-}))\,\lambda(de)\right)dr;\\
(iii)~\overline{Y}^{i;k,t,x_{k}}_{s}\geq \ell(s,X^{k,t,x_{k}}_{s})~\textrm{and}~ \displaystyle\int^{T}_{0}(\overline{Y}^{i;k,t,x_{k}}_{s}- \ell(s,X^{k,t,x_{k}}_{s}))\mathrm{d}\overline{K}^{i;k,t,x_{k}}_{s}=0;
\end{array}
\right.
\end{equation}
and
\begin{equation}\label{5.73}
\left
\{\begin{array}{ll}
(i)~\vec{\overline{Y}}^{k,t,x_{k}}:=(\overline{Y}^{i;k,t,x_{k}})_{i=1,m}\in\mathcal{S}^{2}(\mathbb{R}^{m}),~\overline{Z}^{k,t,x_{k}}:=(\overline{Z}^{i;k,t,x_{k}})_{i=1,m}\in\mathbb{H}^{2}(\mathbb{R}^{m\times d}),\\\overline{K}^{k,t,x_{k}}:=(\overline{K}^{i;k,t,x_{k}})_{i=1,m}\in\mathcal{A}^{2}_{c},~\overline{U}^{k,t,x_{k}}:=(\overline{U}^{i;k,t,x_{k}})_{i=1,m}\in\mathbb{H}^{2}(\mathbb{L}^{2}_{m}(\lambda));\\
(ii)~\overline{Y}^{i;k,t,x_{k}}_{s}= g^{i}(X^{k,t,x_{k}}_{T})+
\overline{K}^{i;k,t,x_{k}}_{T}-\overline{K}^{i;k,t,x_{k}}_{s}-\displaystyle\int^{T}_{s}\overline{Z}^{i;k,t,x_{k}}\mathrm{d}
\mathrm{B}_{r}-\displaystyle\int^{T}_{s}
\displaystyle\int_{\mathrm{E}}\overline{U}^{i;k,t,x_{k}}_{r}(e)\tilde{\mu}(\mathrm{d}r,\mathrm{d}e)\\
\quad\quad\quad+\displaystyle\int^{T}_{s}h^{(i)}\left(r,X^{k,t,x_{k}}_{r},\overline{Y}^{i;k,t,x_{k}}_{r},\overline{Z}^{i;k,t,x_{k}}_{r},
\right.\\

\left.
\qquad\qquad\qquad\qquad\qquad\displaystyle\int_{\mathrm{E}}\gamma^{i}(t,X^{k,t,x_{k}}_{r},e)(\overline{u}_{i}(s,X^{k,t,x_{k}}_{s-}+\beta(s,X^{k,t,x_{k}}_{s-},e))-\overline{u}_{i}(s,X^{k,t,x_{k}}_{s-}))\,\lambda(de)\right)dr;\\
(iii)~\overline{Y}^{i;k,t,x_{k}}_{s}\geq \ell(s,X^{k,t,x_{k}}_{s})~\textrm{and}~ \displaystyle\int^{T}_{0}(\overline{Y}^{i;k,t,x_{k}}_{s}- \ell(s,X^{k,t,x_{k}}_{s}))\mathrm{d}\overline{K}^{i;k,t,x_{k}}_{s}=0. 
\end{array}
\right.
\end{equation}
By proof of step $2$ of proposition $4.4$, $(\overline{Y}^{i;k,t,x}, \overline{K}^{i;k,t,x},\overline{Z}^{i;k,t,x},\overline{U}^{i;k,t,x}1_{\{|e|\geq \frac{1}{k}\}})_{k}$ converge to $(\overline{Y}^{i;t,x}, \overline{K}^{i;t,x},\overline{Z}^{i;t,x},\\\overline{U}^{i;t,x})$ in $\mathcal{S}^{2}(\mathbb{R})\times\mathcal{A}^{2}_{c}\times\mathbb{H}^{2}(\mathbb{R}^{\kappa\times d})\times\mathbb{H}^{2}(\mathbb{L}^{2}(\lambda)) $.\\
Let $((v^{k}_{i=1,m}))_{k\geq 1}$ be the sequence of continuous deterministics functions such that for any $t\leq T$ and $s\in[t,T]$,\\
$$\overline{Y}^{i;k,t,x}_{s}=\overline{v}^{k}_{i}(s,^{k}X^{t,x}_{s})~\text{and } \overline{Y}^{i;k,t,x_{k}}_{s}=\overline{v}^{k}_{i}(s,^{k}X^{t,x_{k}}_{s})~~\forall i=1,\ldots,m.$$
Such that we have respectively by proof of proposition $4.4$ in step $1$ and step $2$:\\
$(i)~\overline{U}^{i;k,t,x}_{s}(e)=(v^{i}(s,^{k}X^{t,x}_{s-}+\beta(s,^{k}X^{t,x}_{s-},e))-v^{i}(s,^{k}X^{t,x}_{s-}))$, $ds\otimes d\mathbb{P}\otimes d\lambda_{k}$-a.e on $[t,T]\times\Omega\times E$;\\
$(ii)~\text{the sequence } ((v^{k}_{i=1,m}))_{k\geq 1}$ converge to $v^{i}(t,x)$
by using (\ref{4.59}).\\
So that $x_{k}\longrightarrow_{k} x$ we take the following estimation which is obtaining by Ito's formula and by the properties of $h^{(i)}$.
\begin{eqnarray}\label{5.74}
& {}{} &\mathbb{E}\left[\left|\vec{Y}^{k,t,x_{k}}_{s}-Y^{k,t,x}_{s}\right|^{2}+\left|K^{k,t,x_{k}}_{T}-K^{k,t,x}_{T}\right|^{2}+\displaystyle\int^{T}_{0}\left\lbrace\left|Z^{k,t,x_{k}}_{s}-Z^{k,t,x}_{s}\right|^{2}
\right.
\right.
\nonumber\\
&{}{}&\qquad\qquad\qquad\left.\left.+\displaystyle\int_{E}\left|U^{k,t,x_{k}}_{s}-U^{k,t,x}_{s}\right|^{2}\,\lambda_{k}(de)\right\rbrace\,ds\right]\nonumber\\
& {}{} &\leq\mathbb{E}\left[\left|g(^{k}X^{t,x_{k}}_{T})-g(^{k}X^{t,x}_{T})\right|^{2}\right]+\mathbb{E}\left[\sup_{s\leq T}\left|\ell(^{k}X^{t,x_{k}}_{s})-\ell(^{k}X^{t,x}_{s})\right|^{2}\right] +C\mathbb{E}\left[\displaystyle\int^{T}_{s}
\left|\vec{Y}^{k,t,x_{k}}_{r}-\vec{Y}^{k,t,x}_{r}\right|^{2}\,dr\right]\nonumber\\
& {}{} &+C\mathbb{E}\left[\displaystyle\int^{T}_{0}\left|^{k}X^{t,x_{k}}_{r}-^{k}X^{t,x}_{r}\right|^{2}(1+\left|^{k}X^{t,x_{k}}_{r}\right|^{p}+\left|^{k}X^{t,x}_{r}\right|^{p})^{2}\,dr\right]\nonumber\\
& {}{} &+C\displaystyle\sum_{i=1,m}\mathbb{E}\left[\displaystyle\int^{T}_{s}\left|\mathrm{B}_{i}\overline{u}^{i}(r,^{k}X^{t,x_{k}}_{r})-\mathrm{B}_{i}\overline{u}^{i}(r,^{k}X^{t,x}_{r})\right|^{2}\,dr\right].\nonumber\\
\end{eqnarray}
Next using (\ref{4.54}) and (\ref{4.55}), the continuty of the function $(t,x)\mapsto \mathrm{B}_{i}\overline{u}^{i}(t,x)$ and the fact of it belong to $\Pi_{g}$ and in the other hand the majoration of the fourth term of (\ref{4.58}); we can use Gronwall's lemma for $s=t$ $\forall i=1,\ldots,m$,\\
$$v^{k}_{i}(t,x_{k})\longrightarrow_{k}v^{k}_{i}(t,x).$$
Therefore by (i)-(ii) we have, for any $i=1,\ldots,m$,\\
\begin{equation}\label{5.75}
\overline{U}^{i;t,x}_{s}(e)=(v^{i}(s,X^{t,x}_{s-}+\beta(s,X^{t,x}_{s-},e))-v^{i}(s,X^{t,x}_{s-}))\quad ds\otimes d\mathbb{P}\otimes d\lambda-\text{a.e. in  }[t,T]\times\Omega\times E, \quad\forall i\in\{1,\ldots, m\}.
\end{equation}
By this result we can replace $(\overline{u}_{i}(s,X^{t,x}_{s-}+\beta(s,X^{t,x}_{s-},e))-\overline{u}_{i}(s,X^{t,x}_{s-}))$ by $\overline{U}^{i;t,x}_{s}(e)$ in (\ref{5.71}), we deduce that the quadriple $(\overline{Y}^{t,x}, \overline{K}^{t,x},\overline{Z}^{t,x},\overline{U}^{t,x})$ verifies: $\forall i\in\{1,\ldots,m\}$ 
\begin{equation}\label{5.76}
\left
\{\begin{array}{ll}
(i)~\vec{\overline{Y}}^{t,x}:=(\overline{Y}^{i;t,x})_{i=1,m}\in\mathcal{S}^{2}(\mathbb{R}^{m}),~\overline{Z}^{t,x}:=(\overline{Z}^{i;t,x})_{i=1,m}\in\mathbb{H}^{2}(\mathbb{R}^{m\times d}),~\overline{K}^{t,x}:=(\overline{K}^{i;t,x})_{i=1,m}\in\mathcal{A}^{2}_{c},\\\overline{U}^{t,x}:=(\overline{U}^{i;t,x})_{i=1,m}\in\mathbb{H}^{2}(\mathbb{L}^{2}_{m}(\lambda));\\
(ii)~\overline{Y}^{i;t,x}_{s}= g^{i}(X^{t,x}_{T})+
\overline{K}^{i;t,x}_{T}-\overline{K}^{i;t,x}_{s}-\displaystyle\int^{T}_{s}\overline{Z}^{i;t,x}_{r}\,\mathrm{d}
\mathrm{B}_{r}-\displaystyle\int^{T}_{s}
\displaystyle\int_{\mathrm{E}}\overline{U}^{i;t,x}_{r}(e)\,\tilde{\mu}(\mathrm{d}r,\mathrm{d}e).\\
\quad\quad\quad+\displaystyle\int^{T}_{s}h^{(i)}(r,X^{t,x}_{r},\overline{Y}^{i;t,x}_{r},\overline{Z}^{i;t,x}_{r},\displaystyle\int_{\mathrm{E}}\gamma^{i}(r,X^{t,x}_{r},e)\overline{U}^{i;t,x}_{r}\,\lambda(de))dr;\\
(iii)~\overline{Y}^{i;t,x}_{s}\geq \ell(s,X^{t,x}_{s})~\textrm{and}~ \displaystyle\int^{T}_{0}(\overline{Y}^{i;t,x}_{s}- \ell(s,X^{t,x}_{s}))\mathrm{d}\overline{K}^{i;t,x}_{s}=0.
\end{array}
\right.
\end{equation}
It follows that $$\forall i\in\{1,\ldots,m\},\quad \overline{Y}^{i;t,x}={Y}^{i;t,x}.$$
With the uniqueness of solution (\ref{5.68}), we have $u^{i}=\overline{u}^{i}=v^{i}$ which means that the solution of (\ref{eq1}) in the sense of Definition $5.1$ is unique inside the class $\mathcal{U}$. 
\end{proof}

\section{Extension}
In this section, we will redefine the function $h^{(i)}$ as a function of $\|U^{i,t,x}\|_{\mathbb{L}^{2}(\lambda)}$ $\forall i\in\{1,\ldots,m\}$.\\
And to show that the results of the previous section remain valid.\\
Let us consider for any $i\in\{1,\ldots,m\}$ the functions $f^{(i)}$, defined by 
$$\forall (t,x,y,z,\zeta)\in[0,T]\times\mathbb{R}^{k}\times\mathbb{R}^{m+d}\times\mathbb{L}^{2}(\lambda);\quad\quad f^{(i)}(t,x,y,z,\zeta)=h^{(i)}(t,x,y,z,\|\zeta\|_{\mathbb{L}^{2}(\lambda)});$$ where the functions $(h^{(i)})_{i=1,m}$ are the sames defined in section $2$.\\
We recall that the result of Theorem $5.2$ is obtained by having mainly $U^{t,x}_{s}(e)=(u^{i}(s,X^{t,x}_{s-}+\beta(s,X^{t,x}_{s-},e))-u^{i}(s,X^{t,x}_{s-}))$; this makes it possible to have the definition 4.1 by passing through a modification of the expression of $B_{i}u^{i}$ $\forall i\in\{1,\ldots,m\}$.\\
We show that $\|U^{i,t,x}_{s}(e)\|^{2}_{\mathbb{L}^{2}(\lambda)}=\|(u^{i}(s,X^{t,x}_{s-}+\beta(s,X^{t,x}_{s-},e))-u^{i}(s,X^{t,x}_{s-}))|\|^{2}_{\mathbb{L}^{2}(\lambda)}$ and that in this case $B_{i}u^{i}$ is well $\forall i\in\{1,\ldots,m\}$.\\
Let now $(t,x)\in[0,T]\times\mathbb{R}^{d}$  and let us consider the following m-dimensional RBSDE with jumps:
\begin{equation}\label{6.76}
\left
\{\begin{array}{ll}
(i)~\vec{Y}^{t,x}:=(Y^{i,t,x})_{i=1,m}\in\mathcal{S}^{2}(\mathbb{R}^{m}),~Z^{t,x}:=(Z^{i,t,x})_{i=1,m}\in\mathbb{H}^{2}(\mathbb{R}^{m\times d}),\\
 K^{t,x}:=(K^{i,t,x})_{i=1,m}\in\mathcal{A}^{2}_{c},~ U^{t,x}:=(U^{i,t,x})_{i=1,m}\in\mathbb{H}^{2}(\mathbb{L}^{2}_{m}(\lambda));\\
\forall i\in\{1,\ldots, m\}~ Y^{i;t,x}_{T}= g^{i}(X^{t,x}_{T})~\text{and};\\
(ii)~dY^{i;t,x}_{s}=-f^{(i)}(s,X^{t,x}_{s},(Y^{i;t,x}_{s})_{i=1,m},Z^{i;t,x}_{s},\|U^{t,x}_{s}(e)\|_{\mathbb{L}^{2}(\lambda)})ds-
\mathrm{d}\mathrm{K}^{i;t,x}_{s}\\
\quad\quad\quad\quad\quad\quad
\quad\quad+Z^{i;t,x}_{s}\mathrm{d}
\mathrm{B}_{s}+\displaystyle\int_{\mathrm{E}}\mathrm{U}^{i;t,x}
_{s}(e)\tilde{\mu}(\mathrm{d}s,\mathrm{d}e),\quad s\leq T;\\
(iii)~Y^{i;t,x}_{s}\geq \ell(s,X^{t,x}_{s})~\textrm{and}~ \displaystyle\int^{T}_{0}(Y^{i;t,x}_{s}- \ell(s,X^{t,x}_{s}))\mathrm{d}\mathrm{K}^{i;t,x}_{s}=0. 
\end{array}
\right.
\end{equation}
By assumed that (\textbf{H1}), (\textbf{H2}) and (\textbf{H3}) are verified and by proposition $3.3$ the (\ref{3.18}), there exists a family of deterministics continuous functions of polynomial growth $(w^{i})_{i=1,m}$ that fact for any $(t,x)\in[0,T]\times\mathbb{R}^{k}$,
$$\forall s\in[t,T],\qquad Y^{i;t,x}_{s}=w^{i}(s,X^{t,x}_{s}).$$
Such that by the same proposition, the family $(w^{i})_{i=1,m}$ is a viscosity solution of the following system: 
\begin{equation}\label{6.77}
\left
\{\begin{array}{ll}
\min\Big\{w^{i}(t,x)-\ell(t,x);-\partial_{t}w^{i}(t,x)-b(t,x)^{\top}\mathrm{D}_{x}w^{i}(t,x)-\frac{1}{2}\mathrm{Tr}(\sigma\sigma^{\top}(t,x)\mathrm{D}^{2}_{xx}w^{i}(t,x))\\
\quad\quad-\mathrm{K}_{i}w^{i}(t,x)-\mathit{h}^{(i)}(t,x,(w^{j}(t,x))_{j=1,m},(\sigma^{\top}\mathrm{D}_{x}w^{i})(t,x),\mathrm{B}_{i}w^{i}(t,x))\Big\}=0,\quad (t,x)\in\left[ 0,T\right] \times\mathbb{R}^{k};\\
w^{i}(T,x)=g^{i}(x).
\end{array}
\right.
\end{equation}
Indeed, using Lemma $4.1$ and the fact that $U^{i;t,x_{k}}$ converges to $U^{i;t,x}$ $\forall i\in\{1,\ldots,m\}$ when $x_{k}\longrightarrow_{k} x$, we deduce that $\|U^{i,t,x_{k}}_{s}(e)\|_{\mathbb{L}^{2}(\lambda)}\longrightarrow_{k} \|U^{i,t,x}_{s}(e)\|_{\mathbb{L}^{2}(\lambda)}$.\\
Moreover, from property of $h^{(i)}$ and the proof of theorem $5.2$ step $2$ (viscosity solution uniqueness), $\|U^{i,t,x}_{s}(e)\|^{2}_{\mathbb{L}^{2}(\lambda)}=\|(w^{i}(s,X^{t,x}_{s-}+\beta(s,X^{t,x}_{s-},e))-w^{i}(s,X^{t,x}_{s-}))\|^{2}_{\mathbb{L}^{2}(\lambda)}$; from where \\$B_{i}w^{i}=\left\lbrace\displaystyle\int_{E}|(w^{i}(s,X^{t,x}_{s-}+\beta(s,X^{t,x}_{s-},e))-w^{i}(s,X^{t,x}_{s-}))|^{2}\,\lambda(de)\right\rbrace^{\frac{1}{2}}$.\\
Thanks to corollary $4.3$, we deduce that $B_{i}w^{i}$ is well defined $\forall i\in\{1,\ldots,m\}$.
\newpage
\section*{Appendix. Barles et al.'s definition for viscosity solution of IPDE (\ref{eq1})}
In the paper by Barles et al. \cite{bar}, the definition of the viscosity solution of the system (\ref{eq1}) is given as follows.
\begin{definition}
We say that a family of deterministics functions $u=(u^{i})_{i=1,m}$ which is continuous $\forall i\in\{1,\ldots,m\}$, is a viscosity sub-solution (resp. super-solution) of the IPDE (\ref{eq1}) if:\\
$(i)\quad \forall x\in\mathbb{R}^{k}$, $u^{i}(x,T)\leq g^{i}(x)$ (resp. $u^{i}(x,T)\geq g^{i}(x)$);\\
$(ii)\quad\text{For any } (t,x)\in[0,T]\times\mathbb{R}^{k}$ and any function $\phi$ of class $C^{1,2}([0,T]\times\mathbb{R}^{k})$ such that $(t,x)$ is a global maximum point of $u^{i}-\phi$ (resp. global minimum point of $u^{i}-\phi$) and $(u^{i}-\phi)(t,x)=0$,  one has
\begin{equation*}
\min\left\lbrace u^{i}(t,x)-\ell(t,x);-\partial_{t}\phi(t,x)-\mathcal{L}^{X}\phi(t,x)-h^{i}(t,x,(u^{j}(t,x))_{j=1,m},\sigma^{\top}(t,x))D_{x}\phi(t,x),B_{i}\phi(t,x))\right\rbrace\leq 0 
\end{equation*}
$\left(resp.
\right.$
\begin{equation*}
\left.
\min\left\lbrace u^{i}(t,x)-\ell(t,x);-\partial_{t}\phi(t,x)-\mathcal{L}^{X}\phi(t,x)-h^{i}(t,x,(u^{j}(t,x))_{j=1,m},\sigma^{\top}(t,x))D_{x}\phi(t,x),B_{i}\phi(t,x)(t,x))\right\rbrace\scriptstyle\geq 0\right).
\end{equation*}
The family $u=(u^{i})_{i=1,m}$ is a viscosity solution of (\ref{eq1}) if it is both a viscosity sub-solution and viscosity super-solution.\\
Note that $\mathcal{L}^{X}\phi(t,x)=b(t,x)^{\top}\mathrm{D}_{x}\phi(t,x)+\frac{1}{2}\mathrm{Tr}(\sigma\sigma^{\top}(t,x)\mathrm{D}^{2}_{xx}\phi(t,x))+\mathrm{K}\phi(t,x)$;\\
where $\mathrm{K}\phi(t,x)=\displaystyle\int_{\mathrm{E}}(\phi(t,x+\beta(t,x,e))-\phi(t,x)-\beta(t,x,e)^{\top}\mathrm{D}_{x}\phi(t,x))\lambda(de)$.
\end{definition}

\end{document}